\documentclass[a4paper,12pt]{amsart}
\usepackage[utf8]{inputenc}
\usepackage[english]{babel}
\usepackage{amsfonts}
\usepackage{amsmath}
\usepackage{mathrsfs}
\usepackage{amssymb}
\usepackage{amsthm}
\usepackage{amscd}
\usepackage{latexsym}
\usepackage{amstext}
\usepackage{amsxtra}
\usepackage[alphabetic,backrefs,lite]{amsrefs} 
\usepackage{color}

\definecolor{darkgreen}{rgb}{0.0, 0.5, 0.0}

\usepackage[all]{xy}
\usepackage{wasysym}

\usepackage{tikz}
\usetikzlibrary{intersections}

\usepackage{geometry}
\usepackage{enumerate}
\usepackage{wasysym}
\usepackage{latexsym} 
\usepackage{colonequals}
\usepackage{ulem}

\usepackage{nicematrix}
\usepackage{tikz}

\usepackage[
        draft=false,
        colorlinks, citecolor=darkgreen,
        backref,       
        pdfauthor={D. Bricalli},
        pdftitle={Linear spaces in Hessian loci of cubic hypersurfaces},
        linktocpage        
]{hyperref}
\hypersetup{citecolor=blue,linktocpage}

\usepackage{dsfont}

\newtheorem{theorem}{Theorem}[section]
\newtheorem{lemma}[theorem]{Lemma}

\newtheorem*{theorem*}{Theorem}
\newtheorem*{conjecture*}{Conjecture}
\newtheorem*{THA}{Theorem A}
\newtheorem*{THB}{Theorem B}
\newtheorem*{THC}{Theorem C}
\newtheorem*{THD}{Theorem D}
\newtheorem{corollary}[theorem]{Corollary}
\newtheorem{proposition}[theorem]{Proposition}
\newtheorem{conjecture}[theorem]{Conjecture}

\newtheorem*{configuration}{Configuration for n=4}

\theoremstyle{definition}
\newtheorem{definition}[theorem]{Definition}

\theoremstyle{remark}
\newtheorem{remark}[theorem]{Remark}

\numberwithin{equation}{section}

\newcommand{\nc}{\newcommand} 
\newcommand{\cA}{\mathcal{A}}

\newcommand{\cB}{\mathcal{B}}

\newcommand{\cC}{\mathcal{C}}
\newcommand{\bC}{\mathbb{C}}
\newcommand{\cD}{\mathcal{D}}

\newcommand{\cH}{\mathcal{H}}

\newcommand{\bK}{\mathbb{K}}
\newcommand{\cL}{\mathcal{L}}

\newcommand{\cP}{\mathcal{P}}
\newcommand{\bP}{\mathbb{P}}
\newcommand{\cQ}{\mathcal{Q}}

\newcommand{\cU}{\mathcal{U}}

\newcommand{\cV}{\mathcal{V}}

\newcommand{\cW}{\mathcal{W}}

\nc{\fr}{{\rightarrow}}
\nc{\co}{{\nabla}}

\nc{\cu}{{\barline{\nabla}}}

\usepackage{mathtools}

\DeclareMathOperator{\rk}{rk}

\DeclareMathOperator{\Ann}{Ann}

\DeclareMathOperator{\Rank}{Rank}
\DeclareMathOperator{\Sing}{Sing}

\DeclareMathOperator{\Bs}{Bs}

\DeclareMathOperator{\diag}{diag}

\DeclareMathOperator{\Sym}{Sym}

\usepackage{tcolorbox}
\newtcolorbox{mybox}{colback=blue!5!white,colframe=blue!75!black}

\begin{document}
\title[Linear spaces in Hessian loci of cubic hypersurfaces]{Linear spaces in Hessian loci of cubic hypersurfaces}

\thanks{
\\
\noindent {\bf Acknowledgements}: \\
The author wants to express his gratitude to Gian Pietro Pirola, for helpful discussion on the topics of this paper and for a remarkable help for the proof of the theorem presented in the last section. The author would also like to thank Filippo Favale for several stimulating and helpful discussions on these topics and for a careful preliminary reading of this paper.  The author is supported by INdAM - GNSAGA}


\author{Davide Bricalli}
\address{Dipartimento di Matematica,
	Universit\`a degli Studi di Pavia,
	Via Ferrata, 5
	I-27100 Pavia, Italy}
\email{davide.bricalli@unipv.it \ bricalli@altamatematica.it}

\maketitle

\begin{abstract}
    In this paper we will study the Hessian hypersurface associated with a smooth cubic. We prove that the existence of a Hessian locus, associated with a smooth cubic form f, of dimension bigger then the expected one forces the cubic f to be of Thom-Sebastiani type. Moreover, we will analyze the existence of some projective linear spaces in such Hessian loci and their nature in terms of the Hessian matrix. From this, we show that the only smooth cubic threefold having the same Hessian variety as the one associated with a general cubic form $f$ of Waring Rank $6$ is $f$ itself. Finally, we prove that the hessian associated with a smooth hypersurface of any degree and dimension is not a cone.
\end{abstract}

\section*{Introduction}
\label{SEC:Intro}
Given an algebraically closed field $\bK$ of characteristic $0$ and a homogeneous polynomial $f\in S_n:=\bK[x_0,\dots,x_n]=\bigoplus_{d\geq0}S^d_n$, one can naturally define the Hessian matrix $H_f$ of $f$ and its hessian polynomial $h_f$, i.e. the determinant of the Hessian matrix $\det(H_f)$. The polynomials having a vanishing hessian have been deeply studied, starting with \cite{Hes59} and \cite{GN}, or more recently by \cite{Rus},\cite{Wa20},\cite{GR},\cite{GR09},\cite{BFP} and others. In the case where $h_f$ is the zero polynomial we will say that $f$ lives in the Gordan-Noether locus $\cC_{GN}$ (the name is in honor of Gordan and Noether and their work \cite{GN}). In this paper, we will focus on the case of smooth cubic hypersurfaces: it is well known that if $V(f)\subset\bP^n$ is a smooth hypersurface, then $f\not\in\cC_{GN}$ and $h_f$ is a homogeneous polynomial of degree $(n+1)(d-2)$.\\
In this paper, we will deal with the case where $d=3$, and in particular with hessian loci associated with a smooth cubic hypersurface $V(f)\subset\bP^n$. Together with the Hessian hypersurface $\cH_f:=V(h_f)\subset\bP^n$, defined as the zero locus of the hessian polynomial $h_f$, one can define other hessian loci
$$\cD_k(f):=\{[x]\in\bP^n \ | \ \rk(H_f(x))\leq k\}\subseteq\bP^n.$$
After the classical studies dealing with plane elliptic curves and their hessians (which are plane cubic curves and generically smooth), a lot of work has been done for higher dimensional cases. For example the Hessian quartic surface associated with a general smooth cubic surface has been studied by \cite{Seg}, \cite{Hut} or more recently by \cite{DVG}, for example. Adler, in \cite{AR}, dealt with the case of cubic threefolds and their associated hessian varieties, while the hessian hypersurface associated with a general smooth cubic fourfold has been analyzed and described (especially in terms of its singularities, as Adler did in the case of threefolds) in \cite{BFP2}, where the authors proved, among other things, that for {\it{any}} homogeneous polynomial $f$ of degree $3$, defining a smooth hypersurface $V(f)\subset\bP^n$, one has
    $$\Sing(\cH_f)=\cD_{n-1}(f), \quad \dim(\Sing(\cH_f))\geq n-3, \quad \mbox{and} \quad \dim(\cD_k(f))\leq k-1.$$
It then follows that for a smooth cubic hypersurface $V(f)$ the singular locus of $\cH_f$ has either dimension $n-3$ or $n-2$, where the former case is realized for $f$ general. It was then natural to ask ourselves whether it is possible to characterize those cubic forms $f$, such that $\dim(\Sing(\cH_f))=n-2$, and in \cite{BFP3}, the authors proved that if $V(f)\subset\bP^n$ for $n\leq5$ is a smooth cubic hypersurface then 
    $$\dim(\Sing(\cH_f))=n-2 \qquad \iff \qquad f \ \mbox{is of Thom-Sebastiani type}.$$
By saying that $f$ is of Thom-Sebastiani type (see the Definition \ref{DEF:TS} in Section \ref{SEC:preliminaries}), we mean that, up to a linear change of coordinates, it can be written as sum of two polynomials depending on disjoint sets of variables. The name comes from the work \cite{TS} of Sebastiani and Thom and \cite{Tho} of Thom.\\
On the other hand, in \cite{BFGre}, it has been show that $f$ has to be of Thom-Sebastiani type, also in the case where the Hessian locus $\cD_1(f)$ is not empty. In other words, $f$ is forced to be of Thom-Sebastiani type, when $\cD_k(f)$ has the maximal allowed dimension, for $k=1$ or $k=n-1$. Here, we ask ourselves whether the same result holds also in the case where another hessian locus $\cD_k(f)$, for $k\neq1,n-1$, has dimension $k-1$. We prove the following:
\begin{THA}
    Let $X=V(f)\subset\bP^n$ be a smooth cubic hypersurface with $n\geq3$. Then 
    $$\dim(\cD_2(f))=1 \quad \mbox{implies that} \quad \mbox{either } f \mbox{ is of Thom-Sebastiani type or } \dim(\cD_{n-1}(f))=n-2.$$
    In particular, for $n\leq5$, if $\dim(\cD_2(f))=1$, then $f$ is of Thom-Sebastiani type.
\end{THA}
Let us stress that the dichotomy presented in the thesis of the first statement is not at all strict. Indeed, one can easily see that if $f$ is of Thom-Sebastiani type, then $\dim(\Sing(\cH_f))=n-2$, and we do actually expect that the converse holds too, for smooth cubic hypersurfaces in every dimension, as proved till the case of cubic fourfolds. The proof essentially relies on the study of particular linear varieties contained in the hessian hypersurfaces and arising as projectivization of the kernel $\iota_f(P):=\bP(\ker(H_f(P))$ with $P\in\cD_2(f)$. More generally, it has been shown (see for example \cite{BFP3}) that the existence of linear subspaces in suitable hessian loci $\cD_k(f)$ leads to remarkable consequences in terms of $f$ itself, as that to be of Thom-Sebastiani type (see Theorem \ref{THM:charactTS}, recalled in Section \ref{SEC:preliminaries}).\\
This fact brought us to a deep study of particular linear subspaces in the Hessian variety: in this paper we analyze the intrinsic nature of $(n-2)$-planes contained in a Hessian variety $\cH_f\subset\bP^n$ associated to a smooth cubic hypersurface $V(f)$. It turns out that, except for a family of cubic forms, which can be explicitly written down, an $(n-2)$-plane contained in $\cH_f$ is necessarily of the form $\iota_f(P)$ for some $P\in\cH_f$:
\begin{THB}
    Let $X=V(f)\subset\bP^n$ be a smooth cubic hypersurface defined by a polynomial $f$ not of Thom-Sebastiani type. Assume there exists $\Pi\simeq\bP^{n-2}$ contained in the Hessian locus $\cH_f$. Then either there exists a point $P\in\cH_f$ such that $\Pi=\iota_f(P)$ or, up to a suitable change of coordinates, $f$ is an element of the explicit linear system $(\star)$.
\end{THB}
The linear system $(\star)$ refers to that described in Equation \eqref{EQ:linearsystem} in Section \ref{SEC:linearspaces} (Theorem \ref{THM:linearspaces}).\\

We will use Theorem B to give an evidence for a conjecture of Ciliberto and Ottaviani dealing with the birationality of the hessian map. Such a map is the natural map which associates to a homogeneous polynomial $f$ of degree $d\geq3$ in $n+1$ variables its hessian polynomial $h_f$, i.e.
$$h_{d,n}:\bP(S^d_n)\dashrightarrow\bP(S^{(d-2)(n+1)}_n).$$
In \cite{CO}, Ciliberto and Ottaviani conjectured that the hessian map $h_{d,n}$ is birational onto its image for every $d\geq3$ and for every $n\geq2$, except for the case where $(d,n)=(3,2)$. It's indeed known that in the case of cubic plane curves the hessian map is generically $3:1$. In \cite{CO}, the authors proved that the hessian map is generically finite (for $d\geq3$ and $n\geq2$) and proved their conjecture for the case of cubic surfaces in $\bP^3$. The case of plane curves has then been analyzed in \cite{Beo} and \cite{COCD}, where the conjecture is proved for $n=2$ and $d\neq5$. Here, by using Theorem B, we prove the following:
\begin{THC}
    Let $[f]\in\bP(S^3_4)$ be a general cubic form of Waring Rank $6$, then
    $$h_{{3,4}_{|\cU}}^{-1}(h_{{3,4}_{|\cU}}(f))=\{[f]\},$$
    where $\cU$ denotes the subset of $\bP(S^3_4)$ of cubic forms defining smooth hypersurfaces in $\bP^4$.
\end{THC}
By saying that $f$ is of Waring Rank $k$, we mean that, up to a change of coordinates, it can be written as $f=\sum_{i=1}^kL_k^3$, where the $L_i$'s are linear forms (see Definition \ref{DEF:WR}). In \cite{CO}[Theorem 6.5], the authors proved that the hessian map $h_{d,n}$ restricted to the locus of forms of Waring Rank $n+2$ is birational onto its image. We enlarge this fact in the case of cubic threefolds: the only smooth cubic threefold having the same Hessian variety as the one associated with a general cubic form $f$ of Waring Rank $6$ is $f$ itself. Probably, this is the same example of a reconstruction already known by Ciliberto and Ottaviani, as the said at the end of their paper \cite{CO}.

The techniques used in the proofs of Theorems B and C are essentially based on two natural identifications of the projective space $\bP^n$. Indeed, one can prove (see for example \cite{BFP2}) that given a smooth cubic hypersurface $V(f)\subset\bP^n$, one can think of the projective space $\bP^n$ as the projectivization $\bP(J^2_f)$, of the Jacobian ideal of $f$, in degree $2$. In this way, a linear space contained for example in the Hessian variety $\cH_f$ can be seen as a linear system of singular quadrics in the jacobian ideal of $f$. This kind of topic, i.e. the study of system of singular quadrics (not necessarily in the Jacobian ideal of a cubic form) has been itself object of interest and study (see for example \cite{Dim} or \cite{BM}).\\
Another fundamental tool in this setting is a second identification of the projective space $\bP^n$, namely the one with the first graded part of the apolar ring associated with a hypersurface $V(f)\subset\bP^n$, which is not a cone. 
Such ring is defined as the graded quotient $\cA_f=\cL/\Ann_{\cL}(f)$, where $\cL$ is the ring of linear differential operators with constant coefficients and $\Ann_{\cL}(f)$ is the annihilator ideal of $f$ (see Section \ref{SEC:preliminaries} for more details).

The apolar ring, as well as the Jacobian ring defined as the quotient of $S$ by the Jacobian ideal of $f$, is one of the main examples of standard Artinian Gorenstein algebras (one can see \cite{bookLef} for a deep treatment of the topic, or also \cite{BFP2},\cite{BFP3},\cite{BF}, for the use of these algebras in this specific setting).\\
Finally we prove the following:
\begin{THD}
    Let $V(f)\subset\bP^n$, for $n\geq 1$ be a smooth hypersurface defined by a polynomial $f$ of degree $d\geq3$. Then 
    $$\cH_f \mbox{ is never a cone}.$$
\end{THD}

After a direct computation, the proof of Theorem D is essentially based on a well-known result dealing with generators of the socle of Jacobian rings of smooth hypersurfaces.\\
\ \\
The structure of the article is the following. In Section \ref{SEC:preliminaries}, we will describe the setting we will work with, also giving the basic definitions and results we will start from. In Section \ref{SEC:TS}, we will prove Theorem A (see Theorem \ref{THM:Sigmaplus}), while in Section \ref{SEC:linearspaces} we will deal with Theorem B and the description of particular linear subspaces in the Hessian variety: it will be divided into two subsections dealing respectively with case of surfaces and that of higher dimension. In Section \ref{SEC:reconstruction}, we will deal with Theorem C (see Theorem \ref{THM:CO}), by splitting the steps of the proof into three subsections. Finally, in Section \ref{SEC:hessiancones}, we prove Theorem D (see Theorem \ref{THM:hessianinjacobians} and Corollary \ref{COR:conclusioncones}).

\section{Preliminaries and notations}
\label{SEC:preliminaries}
In this first section, we present the main objects we will deal with, together with some results which will be crucial in what follows. \\
Let $\bK$ be an algebraically closed field of characteristic zero, and let us denote by 
$$S_n=\bK[x_0,\dots,x_n]=\bigoplus_{d\geq0}S_n^d$$
the polynomial ring with coefficients in $\bK$ in $n+1$ variables, equipped with the usual grading.\\
With a homogeneous polynomial $f$ of degree $d$ (i.e. $f\in S_n^d$), we can naturally associate the following objects:
\begin{itemize}
    \item the {\bf{Hessian matrix}} $H_f$, which is a square symmetric matrix of order $(n+1)$ whose entries are the second partial derivatives of $f$, namely
    $$H_f=[\partial^2f/\partial x_i\partial x_j]_{i,j=0,\dots,n};$$
    \item the {\bf{Hessian polynomial}} $h_f$, defined as the determinant of $H_f$.
\end{itemize}
We say that $f$ is in the Gordan-Noether locus $\cC_{GN}$ if its associated hessian polynomial is the zero polynomial, i.e. $h_f\equiv0$ (see for example \cite{GN} or the more modern \cite{Rus}, \cite{GR},\cite{GR09}, \cite{BFP}, for a study of this topic). If, for example, $f$ defines a smooth hypersurfaces of degree $d$ in $\bP^n$, then it is not an element of $\cC_{GN}$ and, in particular, $h_f$ is a homogeneous polynomial of degree $(d-2)(n+1)$. In this case, we can associate with $f$ also
\begin{itemize}
    \item the {\bf{Hessian variety}} $\cH_f$, the hypersurface defined as the zero locus $V(h_f)\subset\bP^n$ of the hessian polynomial. 
\end{itemize}

More generally, let us also define the so called {\bf{Hessian loci}}
$$\cD_k(f)=\{[x]\in\bP^n \ | \ \rk(H_f(x))\leq k\}\subseteq\bP^n,$$
which give a natural stratification of the whole projective space $\bP^n$ (for example, we clearly have that $\cD_{n+1}(f)=\bP^n$ and $\cD_n(f)=\cH_f$). Given $[x]\in\cD_k(f)$, we can define
\begin{equation}
\label{EQ:iota}
\iota_f([x]):=\bP(\ker(H_f(x)));
\end{equation}
in particular, we have that
$$[x]\in\cD_k(f)\setminus\cD_{k-1}(f) \quad \Longrightarrow \quad \iota_f([x])\simeq\bP^{n-k}.$$
Let us now introduce also an incidence correspondence
\begin{equation}
\label{EQ:Gamma}
    \Gamma_f=\{([x],[y])\in\bP^n\times\bP^n \ | \ H_f(x)\cdot y=0, \ \mbox{i.e. }[y]\in\iota_f([x])\}\subset\bP^n\times\bP^n
\end{equation}
and $\pi_1$ and $\pi_2$ will denote the two projections from $\Gamma_f$.
\bigskip
\ \\
From now on, we set $d=3$. In this case, one can observe that the incidence correspondence just defined allows us to identify the Hessian hypersurface $\cH_f$ with the Steinerian hypersurface described by Dolgachev in \cite{Dol}.\\
Let us also denote
$$\cU_n:=\{[f]\in\bP(S^3_n) \ | \ V(f)\subset\bP^n \ \mbox{is a smooth cubic hypersurface}\}.$$
For cubic polynomials, we proved the following facts (see \cite{BFP2}):
\begin{theorem}
    \label{THM:preliminaries}
    Let $[f]\in\cU_n$, then:
    \begin{enumerate}
        \item $\iota_f$ defines a birational involution on $\cH_f$;
        \item $\cH_f$ is reduced, and its singular locus $\Sing(\cH_f)$ coincide with $\cD_{n-1}(f)$;
        \item For any $k$, we have that $\dim(\cD_k(f))\leq k-1$.
    \end{enumerate}
\end{theorem}

In particular, for any smooth cubic hypersurface $V(f)\subset\bP^n$, we proved that 
$$\dim(\Sing(\cH_f))\in\{n-3,n-2\},$$
where the first case is the general one. \\

In \cite{BFP3}, the authors analyzed the case where the dimension of the singular locus of $\cH_f$ is not the expected one and we proved the following:

\begin{theorem}[Theorem 4.1, \cite{BFP3}]
\label{THM:TS}
Let $[f]\in\cU_n$. If $n\leq5$, then the following holds:
$$\dim(\Sing(\cH_f))=n-2 \quad \iff \quad f \ \mbox{ is of Thom-Sebastiani type}.$$
\end{theorem}
By saying that a polynomial is of Thom-Sebastiani type we mean the following:

\begin{definition}
\label{DEF:TS}
    A polynomial $f\in S^d_n\setminus\{0\}$ is said to be of {\it Thom-Sebastiani type} (we will often use the abbreviation TS) if up to a change of coordinates it can be written as
    $$f=f_1(x_0,\dots,x_k)+f_2(x_{k+1},\dots,x_n)$$
for suitable $0\leq k\leq n-1$, and $f_1,f_2$ polynomials of degree $d$ in $k+1$ and $n-k$ variables respectively.\\
A polynomial $f$ of TS type is said to be {\it{cyclic}} if $k=0$, i.e. it can be written as $f=x_0^3+f_2(x_1,\dots,x_n)$.
\end{definition}

In literature, these polynomials are also called {\it{direct sums}}, as for example in \cite{BBKT} or \cite{Fed}.

\begin{remark}
    \label{REM:cubicsTS}
    Let us stress that, for $n\leq4$, saying that a polynomial is of TS type is equivalent for it to be cyclic. For cubic fourfolds, on the other hand, also the case where $f$ can be written as sum of two cubic forms each depending on three variables arises.
\end{remark}

The following characterization of Thom-Sebastiani polynomials has been one of the main tools in the proof of the above Theorem. We recall it here, since it will be useful in the following.

\begin{theorem}[Theorem 2.3, \cite{BFP3}]
\label{THM:charactTS}
 A polynomial $f$ defining a smooth cubic hypersurface is of Thom-Sebastiani type of the form $f=f_1(x_0,\dots,x_k)+f_2(x_{k+1},\dots,x_n)$ if and only if $\cD_{k+1}(f)$ contains a projective space $\bP^{k}$.
\end{theorem}
In \cite{BFGre}, with the same spirit as that of Theorem \ref{THM:TS}, the authors analyzed the case where the Hessian locus $\cD_1(f)$ has the maximal allowed dimension (for $V(f)$ being smooth), i.e. $\dim(\cD_1(f))=0$, and the 
following result has been proved:
\begin{proposition}[Proposition 3.2, \cite{BFGre}]
    \label{PROP:Grenoble}
    Let $[f]\in\cU_n$ be a cubic form. Then it holds:
    $$\dim(\cD_1(f))=0 \ \Longleftrightarrow \ f \ \mbox{ is cyclic}.$$
\end{proposition}

In the following subsection, we will present two natural identifications of the projective space $\bP^n$, which will play a central role in this paper. 

\subsection{Jacobian and apolar rings}
\label{SUBS:jacobianandapolar}
Given a polynomial $f\in S^d_n$, one can naturally define the Jacobian ideal $J_f$ of $f$, which is the ideal in $S_n$ generated by the partial derivatives of $f$, and the {\bf{Jacobian ring}} $R_f$ of $f$: respectively
$$J_f=(\partial f/\partial x_0,\dots,\partial f/\partial x_n) \qquad \mbox{and} \qquad R_f=S_n/J_f.$$
In the case where $V(f)\subset\bP^n$ is a smooth hypersurface of degree $d$, the Jacobian ideal of $f$ is generated (in degree $d-1$) by a regular sequence of degree $d-1$ polynomials, i.e. one has that $\cap_{i=0}^nV(\partial f/\partial x_i)=\emptyset$. In this case, the Jacobian ring $R_f$ is an example of a standard Artinian Gorenstein algebra (see for example \cite{bookLef}, for a deep treatment of this topic) and it can be written as $R_f=R_f^0\oplus\cdots R_f^N$, where $R_f^N$ is called the socle of $R_f$, and $N$ its socle degree. Let us now recall some basic facts, which will be useful later on:
\begin{proposition}
    \label{PROP:soclejacobian}
    Let $f\in S^d_n$ be a homogeneous polynomial of degree $d$ defining a smooth hypersurface $V(f)\subset\bP^n$. Then:
    \begin{itemize}
        \item the socle of the Jacobian ring $R_f$ has degree $N=(n+1)(d-2)$ and $\dim_{\bK}R^N_f=1$;
        \item $h_f\not\in(J_f)$, i.e. the class of $h_f$ is a generator of the socle $R^{(n+1)(d-2)}$.
    \end{itemize}
\end{proposition}

For the second statement, one can refer for example to \cite{AGV}[Section 5.11] or also \cite{DGI}[Proposition 3.6]. Let us observe that this last statement does not hold anymore as soon as the hypersurface $V(f)$ has isolated singularities (see again \cite{DGI}).

Let us now define $\cL_n=\bK[y_0,\dots,y_n]=\oplus_{i\geq0}\cL_n^i$, as the ring of linear differential operators with constant coefficients, where 
$$y_i:=\frac{\partial}{\partial x_i} \ \mbox{for all } i=0,\dots,n.$$
Given a homogeneous polynomial $f\in S^d_n$, let us also introduce the annihilator ideal of $f$ and the {\bf{apolar ring}} of $f$ (another example of standard Artinian Gorenstein algebras) as, respectively,
$$\Ann_{\cL_n}(f):=\{\delta\in\cL_n \ | \ \delta(f)=0\} \qquad \mbox{and} \qquad \cA_f=\cL_n/\Ann_{\cL_n}(f)=A_f^0\oplus\dots\oplus A_f^d.$$
Observe now that, in the case where $V(f)\subset\bP^n$ is not a cone, there is no first derivative which vanishes identically, i.e. $(\Ann_{\cL}(f))_1=(0)$ and so $\dim(A^1_f)=n+1$.\\
From the natural pairing $S_n\times\cL_n\rightarrow S_n$, one deduces the natural identification
$$\bP^n\simeq\bP((S_n^1)^*)\simeq\bP(A_n^1):$$
we will identify the points of the projective space $\bP^n$ as linear differential operators of degree $1$.\\
Let us focus on the case $d=3$. In \cite{BFP2} and \cite{BFP3}, we proved the following facts:
\begin{lemma}
    \label{LEM:apolarfacts}
    Given a cubic hypersurface $X=V(f)\subset\bP^n$, which is not a cone, under the identification $\bP^n\simeq\bP(A_f^1)$, we have that:
    \begin{enumerate}[(a)]
        \item computing $H_f(x)\cdot y$ for $[x],[y]\in\bP^n$ is equivalent to compute $x\cdot y\in A_f^2$;
        \item $X=\{[x]\in\bP(A^1_f) \ | \ x^3=0\}$;
        \item $\Sing(X)=\{[x]\in\bP(A_f^1) \ | \ x^2=0\}$;
        \item $\cH_f=\{[x]\in\bP(A^1_f) \ | \ \exists \ [y]\in\bP(A^1_f) \ \mbox{with } \ xy=0\}$;
        \item for all $x,y\in\bK^{n+1}$ we have $y^t\cdot H_f(x)\cdot y =2(x(f))(y)$.
    \end{enumerate}
\end{lemma}

To conclude this first section, let us make some few comments about the above lemma. From point (a), we see that the incidence correspondence $\Gamma_f$ introduced above (see \eqref{EQ:Gamma}) can also be written as
$$\Gamma_f=\{([x],[y])\in\bP(A_f^1)\times\bP(A^1_f) \ | \ xy=0\}.$$
It is then clear that $\Gamma_f$ is symmetric, in the sense that if $([x],[y])\in\Gamma_f$ then also $([y],[x])\in\Gamma_f$, and so both the projections $\pi_i$ from $\Gamma_f$ are surjective onto the Hessian variety $\cH_f$. With the expression {\it{by symmetry}} in what follows we will mean indeed that 
$$H_f(x)\cdot y=H_f(y)\cdot x.$$
Furthermore, from part (c) of the Lemma above, we get the following property
\begin{equation}
    \label{EQ:notdiagonal}
    V(f)\subset\bP^n \ \mbox{smooth} \quad \iff \quad \Gamma_f\cap\Delta_{\bP^n}=\emptyset \quad \iff \quad \not\exists [x]\in\bP^n \mbox{ such that } [x]\in\iota_f([x]).
\end{equation}

\begin{remark}
     Assume that the points $P_0,\dots,P_n$ give a system of coordinates in $\bP^n$, hence corresponding to a basis $\{y_0,\dots,y_n\}$ in $A_f^1$. Given a cubic form $f$, let us observe that the entries of the Hessian matrix $H_f$ are linear forms. The evaluation of these entries in the point $P_i$ is equivalent to consider the first partial derivative with respect to $x_i$ of the linear form itself.\\
     In other words, for example for $i=0$, we have
     $$H_f(P_0)=\left[\frac{\partial^3f}{\partial x_i\partial x_j \partial x_0}\right]_{i,j=0,\dots,n+1}.$$
     Hence, if for example the $(0,0)-$entry of $H_f(P_0)$ is zero, we have that $y_0^3=0$, i.e. $P_0\in V(f)$, and so the monomial $x_0^3$ does not appear in the expression of $f$, in such a system of coordinates.
\end{remark}

Finally, from the above Lemma \ref{LEM:apolarfacts} (e), one sees that, up to the evaluation of the Hessian matrix of a cubic form $f$, a point $[x]\in\bP^n$ can be thought as of the first partial derivative of $f$ with respect to $x$ itself, which defines a quadric, as element of $\bP(J^2_f)$. In other words, (see also \cite{BFP2}) we have another natural identification of the projective space $\bP^n$, i.e.
$$\bP^n\simeq\bP(J_f^2).$$
Observe then that given a point $[x]\in\cD_k(f)\setminus\cD_{k-1}(f)$, it corresponds to a quadric in $\bP(J_f^2)$ of rank $k$, i.e. to a quadric cone in $\bP^n$, whose vertex coincides exactly with the kernel $\iota_f([x])$ introduced above in Equation \eqref{EQ:iota}.\\

\section{Hessian loci of maximal dimension}
\label{SEC:TS}

Let $[f]\in\bP(S^3_n)$ be a cubic polynomial in $n+1$ variables, whose associated hypersurface $X=V(f)\subset\bP^n$ is smooth, i.e. $[f]\in\cU_n$. As recalled in Section \ref{SEC:preliminaries} (see Theorem \ref{THM:preliminaries}), in this case there is an upper bound for the dimension of the Hessian loci: $\dim(\cD_k(f))\leq k-1$.\\
Moreover, it is known that $f$ has to be of Thom-Sebastiani type if such an upper bound is obtained for $k=0$ (Proposition \ref{PROP:Grenoble}) or for $k=n-1$, with $n\leq5$ (Theorem \ref{THM:TS}).

In this section, we generalize this kind of results, by proving Theorem A: 

\begin{theorem}
\label{THM:Sigmaplus}
    Let $[f]$ be an element of $\cU_n$, with $n\geq3$. Then 
    $$\dim(\cD_2(f))=1 \quad \mbox{implies that} \quad \mbox{either }f \mbox{ is of Thom-Sebastiani type or }\dim(\cD_{n-1}(f))=n-2.$$
    In particular, for $n\leq5$, with the same hypotheses, $f$ is of Thom-Sebastiani type.
\end{theorem}
 Observe that the claim is not true for $n=2$: indeed in this specific case we deal with a smooth cubic plane curve $V(f)$ and $\cD_2(f)$ coincide with the Hessian curve $\cH_f$, which is always of dimension $1$. If we suppose that $f$ is not of Thom-Sebastiani type, i.e. not of Fermat type, it's not true that there exists a point in the Hessian locus $\cD_1(f)=\Sing(\cH_f)$, since $\cH_f$ is generically smooth.
\begin{proof}
    First of all, let us observe that, from above, there is nothing to prove in the case where $n=3$. Let us then assume $n\geq4$.\\
    Let us assume that $f$ is not of Thom-Sebastiani type: we want to prove that $\dim(\cD_{n-1})=n-2$.\\
    Let $C$ be a component of $\cD_2(f)$ with $\dim(C)=1$. Since $f$ is not of TS type, we can assume (by Theorem \ref{THM:charactTS}) that $C\not\simeq\bP^1$ and that $\cD_1(f)=\emptyset$  (by Proposition \ref{PROP:Grenoble}). This last assumption implies that for any point $P\in C$ we have that $\Pi_P:=\iota_f(P)\simeq\bP^{n-2}$. Observe that if $P_1$ and $P_2$ are two distinct points of $C$, then necessarily $\Pi_{P_1}$ and $\Pi_{P_2}$ are distinct. Indeed, if by contradiction $\Pi:=\Pi_{P_1}=\Pi_{P_2}\simeq\bP^{n-2}$ then, by symmetry, for any $Q\in\Pi$ we have that $\iota_f(Q)\supset\bP(\left\langle P_1,P_2\right\rangle)$, i.e. $\Pi\subset\cD_{n-1}(f)$ and so again by Theorem \ref{THM:charactTS}, $f$ would be of TS type, against our assumptions. For the same reason, one can see that for any point $P\in C$, the $(n-2)$-plane $\Pi_P$ is not contained in the singular locus $\cD_{n-1}(f)$.\\
    Let us now fix a point $P\in C$ and, up to a change of coordinates, let us write it as $P=P_0=[1:0:\dots:0]$. Consider also the associated $(n-2)$-plane $\Pi_P$ and $R_1,\dots,R_{n-1}$ general points of $\Pi_P$ in general position. Let us then write the Hessian matrix of $f$ evaluated at the general point of $\Pi_P$ as
$$H_{f_{|\Pi_P}}=H_f\left(\sum_{i=1}^{n-1}\lambda_iR_i\right)=\begin{pmatrix}
    0&\underline{0}\\\underline{0}&M_P)
\end{pmatrix}$$
   since by construction $P_0$ belongs to $\iota_f(R)$ for any $R\in\Pi_P$ and where $M_P$ is an $n\times n$ matrix, whose entries are depending on the $\lambda_i$'s. Since, $\Pi_P\not\subset\cD_{n-1}(f)$ as observed above, we have that the determinant $\det(M_P)$ is not identically zero. Hence, we can define 
   $$Z_P:=V(\det(M_P))\subseteq\cD_{n-1}(f),$$
   which is a subvariety of $\Pi_P$ of dimension $n-3$.\\
   We then have two possibilities: either $Z_P$ is constant with $P$ varying in $C$, or it is not. In the first case, we would have that there exists a variety $Z$ of dimension $n-3$ such that $Z=Z_P$ for $P\in C$ general. But then $Z$ would be contained in distinct $(n-2)$-planes: we have $Z\simeq\bP^{n-3}$. Moreover, being contained in $\Pi_P=\iota_f(P)$ for $P\in C$ general, by symmetry we would have that for any $Q\in Z$, $\iota_f(Q)\supset C$, i.e. $\iota_f(Q)\simeq \bP^s$, with $s\geq2$. This implies that $Z\simeq\bP^{n-3}\subseteq\cD_{n-2}(f)$: $f$ would be of TS type by Theorem \ref{THM:charactTS}, against our assumptions.\\
   The only possibility is that $Z_P$ varies with $P$, i.e. $\dim(\bigcup_{P\in C}Z_P)=n-2$. Since by construction, $\bigcup_{P\in C}Z_P\subseteq\cD_{n-1}(f)$, we have the claim.\\
   The second and last statement follows immediately from Theorem \ref{THM:TS}.
\end{proof}
 Observe that the dichotomy in the thesis of the first statement is not necessarily strict: we actually think that the two possible theses are equivalent to each other (as Theorem \ref{THM:TS} shows for $n\leq5$).

\begin{remark}
    Let us observe that the converse does not hold. Moreover, it is not true in general that given a smooth cubic hypersurface $V(f)\subset\bP^n$, if there is a Hessian loci $\cD_k(f)$ of maximal dimension, i.e. $k-1$, then the same holds for the Hessian loci of ``higher level'', as $\cD_{k+1}(f)$. The following example gives an evidence to both these statements.\\
    It is known (see \cite{Seg},\cite{CO},\cite{DVG}...) that for the general smooth cubic surface $V(g)\subset\bP^3$, then its associated hessian variety is singular in $10$ distinct points, say $Q_1,\dots,Q_{10}$, belonging then to $\cD_2(g)$ and one has $\cD_1(g)=\emptyset$, as $g$ is general. Let us now consider a cyclic cubic threefold $V(f)\subset\bP^4$, defined by
    $$f(x_0,x_1,x_2,x_3,x_4)=x_0^3+g(x_1,x_2,x_3,x_4),$$
    with $g$ as above. Since it is cyclic, one can see that 
    $$\cD_1(f)=\{P_0\}, \mbox{ i.e. } 0=\dim(\cD_1(f))=1-1.$$
    On the other hand, one can see that
    $$\cD_2(f)=\{P_0,Q_1,\dots,Q_{10}\}, \mbox{ i.e. } \dim(\cD_2(f))=0.$$
\end{remark}

\section{Linear spaces in Hessian loci}
\label{SEC:linearspaces}

In this section, we focus our attention on linear subvarieties contained in the Hessian loci associated with a smooth cubic hypersurface and we prove Theorem B from the Introduction.\\
First of all, we introduce a family of cubic hypersurfaces, which will play a key role in what follows. This family is defined by the linear system
\begin{equation}
    \label{EQ:linearsystem}
    (\star) \qquad |\Sym^3W+\left\langle x_{n-2}^3\right\rangle+x_n^2\cdot\left\langle x_0\dots,x_{n-1}\right\rangle+x_n\cdot(\Sym^2W+\alpha x_{n-2}^2+\beta x_{n-1}^2+\gamma x_{n-3}x_{n-2})|,
\end{equation}
where $W=\left\langle x_0,\dots, x_{n-3}\right\rangle$.

\begin{remark}
    \label{REM:familysmooth}
    Observe that, by a standard Bertini argument, one can easily prove that the general element of the above linear system $(\star)$ defines a smooth hypersurface in $\bP^n$, with $n\geq3$.
\end{remark}

\begin{theorem}
\label{THM:linearspaces}
    Let $X=V(f)\subseteq\bP^n$, with $n\geq2$ be a smooth cubic hypersurface defined by a polynomial not of Thom-Sebastiani type. Assume that there exists $\Pi\simeq\bP^{n-2}$ contained in the Hessian locus $\cH_f$, then 
\begin{itemize}
    \item either there exists a point $P\in\cH_f$ such that $\Pi=\iota_f(P)$,
    \item  or, up to a suitable change of coordinates, $f$ can be written as an element of the linear system $(\star)$.
\end{itemize}
\end{theorem}

First of all, observe that the statement holds trivially for the case of curves in $\bP^2$: here, the linear space $\Pi$ from the statement is a point $Q\in\cH_f$, for which there exists another point $P\in\cH_f$ such that $Q=\iota_f(P)$. For the proof we will consider then $n\geq3$.

\begin{remark}
    The assumption that $f$ needs not to be of Thom-Sebastiani type will simplify the proof and the construction for arbitrary $n$. However, for low dimensions ($n\leq5$), we will show that an analogous result holds also for polynomials of TS type.\\
    Moreover, observe that such an $(n-2)$-plane $\Pi$ contained in $\cH_f$ can't be contained in $\Sing(\cH_f)$, otherwise by Theorem \ref{THM:charactTS} $f$ would be of TS type, against the assumptions.
\end{remark}

Let us see that something can be said also in the case where $f$ is of Thom-Sebastiani type:
\begin{remark}
\label{REM:Fermat}
    In the case where $f$ is a polynomial of TS type, it can happen that the $(n-2)$-plane considered in the statement of the above Theorem \ref{THM:linearspaces} is only strictly contained in a kernel $\iota_F(P)$ for some $P\in\cH_f$. For example, let $f=\sum_{i=0}^nx_i^3$ be the Fermat cubic polynomial in $n+1$ variables. If there exists an $(n-2)$-plane $\Pi\simeq\bP^{n-2}$ contained in the Hessian variety $\cH_f$, then there exists a point $P\in\cH_f$ such that $\Pi\subseteq\iota_f(P)$. This is obvious by observing that $\cH_f$ consists of the union of the $n+1$ coordinate hyperplanes and for each $i=0,\dots,n$ we have $V(x_i)=\iota_f(P_i)$, where $P_i$ is the $i-$th coordinate point. Then if $\Pi$ is not contained in $\Sing(\cH_f)$, it is contained in just one of the hyperplanes above and so the claim follows. If it is contained in the singular locus of $\cH_f$, it belongs to the intersection of two kernels, namely $\iota_f(P_i)$ and $\iota_f(P_j)$ for some $i,j\in\{0,\dots,n\}$, with $i\neq j$.
\end{remark}

The following Lemma deals with the case of $(n-2)$-planes contained in the singular locus of Hessian varieties associated with cyclic cubic forms in $\cU_n$.
\begin{lemma}
\label{LEM:TSSing}
    Let $n\leq5$ and let $f$ be a cubic cyclic polynomial in $\cU_n$. If there exists an $(n-2)$-plane $\Pi\simeq\bP^{n-2}$ contained in $\Sing(\cH_f)$, then there exists a point $P\in\cH_f$ such that $\Pi\subseteq\iota_f(P)$.
\end{lemma}

\begin{proof}
    Clearly, there is nothing to prove for $n=2$: let us set $n\in\{4,5\}$ for the moment.\\
    We can write $f=x_0^3+f_1(x_1,\dots,x_n)$: the Hessian hypersurface $\cH_f$ consists of the hyperplane $V(x_0)$ and of the cone $\tilde{\cH}_{f_1}$ of $\cH_{f_1}$ over the point $P_0=[1:0:\dots:0]$ (thinking of $V(f_1)$ as a hypersurface in $\bP^{n-1}$). Observe that $V(x_0)=\iota_f(P_0)$. One can then see that $\Sing(\cH_f)=\cD_{n-1}(f)$ consists of the intersection $V(x_0)\cap\tilde{\cH}_{f_1}$ and of the joint variety $J(P_0,\Sing(\cH_{f_1}))$. Let us then assume that we have $\Pi\simeq\bP^{n-2}$ contained in $\Sing(\cH_f)$: if $\Pi\subseteq V(x_0)\cap\tilde{\cH}_{f_1}$ then in particular it is contained in $V(x_0)=\iota_f(P_0)$, hence the claim follows. Assume then that $\Pi$ is contained in $J:=J(P_0,\Sing(\cH_{f_1}))$: from this it follows that $f_1$ has to be of TS type too. Indeed, otherwise, $\dim(\Sing(\cH_{f_1}))=n-4$ and so $\dim(J)=n-3$: it can't contain an $(n-2)$-plane. Then, since $f_1$ depends on $n\leq5$ variables, the only possibility is that $f_1$ is cyclic. Hence, we have that $\dim(\Sing(\cH_{f_1}))=n-3$, and so $\dim(J)=n-2$: the $(n-2)$-plane is then a component of $J$. We reduce ourselves to look for a $(n-3)-$plane in $\Sing(\cH_{f_1})$, for $f_1$ cyclic. By iterating the argument one reduces to the case where $n=3$. Here, we can write $f=x_0^3+f_1(x_1,x_2,x_3)$; we have then two possibilities: either $f$ is the Fermat polynomial (if $f_1$ is itself cyclic) or $\Sing(\cH_{f_1})=\emptyset$. In the first case, the claim follows from Remark \ref{REM:Fermat}; while in the second case, a line contained in $\Sing(\cH_f)$ has to be contained in $V(x_0)\cap\tilde{\cH}_{f_1}$: this means that $\cH_{f_1}$ is reducible and so $f_1$ is cyclic, i.e. $f$ is again of Fermat type.
\end{proof}


We will now prove the above Theorem \ref{THM:linearspaces} in the following two subsections, where we will split up the case of surfaces and the cases of higher dimension. The technique for the proof of the two cases is similar, but, for $n\geq4$, we will identify a peculiar $(n-4)$-plane, which will be crucial and which does not appear in the case of cubic surfaces in $\bP^3$.

\subsection{Surfaces}
\label{SUBS:Surfaces}
Let us start by dealing with the study of lines contained in the Hessian variety of cubic surfaces in $\bP^3$, and in particular, with the cyclic case.

\begin{lemma}
\label{LEM:cyclicsurfaces}
    Let $V(f)\subset\bP^3$ be a smooth cubic surface defined by a cyclic polynomial $f$. If there exists a line $L\subseteq\cH_f$ then there exists a point $P\in\cH_f$ such that $L\subseteq\iota_f(P)$.
\end{lemma}

\begin{proof}
    Let us write $f=x_0^3+f_1(x_1,x_2,x_3)$. By Remark \ref{REM:Fermat} we can assume that $f_1$ is not of Fermat type, i.e. not of TS type: in particular $\cH_{f_1}$ is a curve of degree $3$, irreducible and smooth. We have, as above, that $\cH_f=V(x_0)\cup\tilde{\cH}_{f_1}$, where $V(x_0)=\iota_f(P_0)$, with $P_0=[1:0:0:0]$, and $\tilde{\cH}_{f_1}$ is the cone of $\cH_{f_1}$ over $P_0$. Let $L$ be a line contained in $\cH_f$: if $L\subset V(x_0)$, we get the claim. Assume then that $L\subset\tilde{\cH}_{f_1}$: since by construction the section $\cH_{f_1}$ is smooth and irreducible, the only possibility is that there exists a point $Q\in\cH_{f_1}$ such that $L$ is the line joining $Q$ and $P_0$. Observe now that considering the curve $V(f_1)$, there exists a point $P\in\cH_{f_1}$ such that $\iota_{f_1}(P)=Q$. In particular, we have that $L=\iota_f(P)$, as claimed.
\end{proof}

Let us now consider the general case and let us prove that Theorem \ref{THM:linearspaces} holds for $n=3$:

\begin{proposition}[Theorem \ref{THM:linearspaces}, for $n=3$]
\label{PROP:surfacecase}
    Let $S=V(f)\subseteq\bP^3$ be a smooth cubic surface not of Thom-Sebastiani type. Let $L\simeq\bP^1$ be a line contained in $\cH_f$, then
    \begin{itemize}
        \item either there exists a point $P\in\cH_f$ such that $L=\iota_f(P)$, 
        \item or, up to a suitable change of coordinates, $f$ can be written as
        \begin{equation}
\label{EQ:specialsurface}
    \tilde{f}=\alpha_0x_0^3+\alpha_1x_1^3+x_3(\beta_0x_0^2+\beta_1x_1^2+\beta_2x_2^2+\beta_3x_0x_1)+x_3^2(\gamma_0x_0+\gamma_1x_1+\gamma_2x_2).
\end{equation}
    \end{itemize}
\end{proposition}

Observe that $\tilde{f}$ is exactly the general element of the system $(\star)$ (see Equation \eqref{EQ:linearsystem}) of Theorem \ref{THM:linearspaces}.

\begin{proof}
    First of all, observe that the line $L$ is not contained in $\Sing(\cH_f)=\cD_2(f)$. Indeed, otherwise, by Theorem \ref{THM:TS}, $f$ would be of TS type, against the assumptions.\\
    If we then fix a line $L$ contained in the Hessian surface $\cH_f$, we have that its general point $P$ is smooth for $\cH_f$ and $\iota_f(P)$ consists of a single point. Let us parametrize $L$ as $A+tB$, where $A$ and $B$ lie in $\cH_f\setminus\cD_2(f)$. For every $t\in\bK$ there exists a point $v(t)$, still in the Hessian $\cH_f$, which belongs to the kernel of the Hessian matrix $H_f$ evaluated in the point $A+tB$, i.e.  
    $$H_f(A+tB)\cdot v(t)=0 \qquad \mbox{or equivalently} \qquad (A+tB)(v(t))=0.$$
    Observe that such element $v(t)$ is unique (up to scalar) if the corresponding point $A+tB\not\in\cD_2(g)$. We can then construct a dominant rational map $\tilde{\varphi}:L\dashrightarrow\cH_f$, mapping a point $A+tB$ to $\iota_f(A+tB)=[v(t)]$. If $\tilde{\varphi}$ is constant, then by symmetry we get the claim: if $\tilde{\varphi}(L)=\{Q\}$, then we would have that $L\subseteq\iota_f(Q)$. In particular, we would have $L=\iota_f(Q)$: indeed, if the inclusion is strict, we have that $\iota_f(Q)\simeq\bP^2\subseteq\cH_f$, i.e. $\cH_f$ is reducible and so $f$ is of TS type (by Theorem \ref{THM:TS}), which is against the assumptions.\\
    Let us then assume that this is not the case, i.e. that the closure of the image is a rational curve $\cV$ and consider the extension $\varphi:L\rightarrow\cV$. We want to prove that $f$ can be written as in Equation \eqref{EQ:specialsurface}.
    \smallskip
    \ \\
    {\underline{Claim: $\cV$ is a projective line}}. Recalling the interpretation of the projective space $\bP^3$ as $\bP(J_f^2)$, where $J_f=( \partial_{x_0}f,\dots,\partial_{x_3}f)$ is the Jacobian ideal of $f$ (see Subsection \ref{SUBS:jacobianandapolar} and in particular Lemma \ref{LEM:apolarfacts}) (e), we have that $L$ represents a linear system of quadrics generically singular in a point (the vertex corresponding to $[v(t)]$). By Bertini's Theorem, the quadric defined by the general point of $L$ is smooth outside the base locus $\cB=\mathrm{Bs}(|L|)$. In other words, the curve $\cV$ consisting of the vertices of these quadrics is contained in the base locus $\cB$.\\
    Since $f$ is smooth, its Jacobian is generated by a regular sequence of quadrics: two distinct point of $L$ define two quadrics intersecting in a curve: $\cV$ is an irreducible component of such intersection. One can then easily observe that the line joining two general points $v_1$ and $v_2$ of $\cV$ is in the intersection of the two quadrics $Q_1$ and $Q_2$ having those points as vertices, i.e. $\bP(\left\langle v_1,v_2\right\rangle)\subset\cB$. We would have that the union of the secant lines of $\cV$ would be contained in $\cB$: the only case where the assumption $\dim(\cB)=1$ is satisfied is the one where $\cV$ is a line, as claimed.\\
    Let us now study the mutual position of the two lines $\cV$ and $L$.
    \smallskip
    \ \\
    {\underline{Claim: The intersection between $L$ and $\cV$ consists of a single point}}. First of all, let us consider the case where $L$ and $\cV$ coincide. We can construct the subvariety $\Gamma_{f_{|L}}$ of $\Gamma_f$ (recall it from \eqref{EQ:Gamma}), defined as the closure of $\pi_1^{-1}(L\setminus\cD_2(f))$, i.e.
    $$\Gamma_{f_{|L}}=\overline{\{([x],[y]) \ | \ [x]\in L\setminus\cD_2(f), \ H_f(x)\cdot y=0\}}.$$
    Since the kernel of the general point of $L$ lives in $\cV$, and so in $L$, we have that $\Gamma_{f_{|L}}$ can be seen as a subvariety of $L\times L\simeq\bP^1\times\bP^1$. Moreover, since by construction we have that $\dim(\Gamma_{f|L})=1$, the intersection $\Gamma_{f_|L}\cap\Delta_{\bP^1}$ is then not trivial, contradicting the smoothness of $f$ (by property \eqref{EQ:notdiagonal}).\\
    Let us now consider the case where $L\cap\cV=\emptyset$ and fix a system of coordinates of $\bP^3$. Let us set $A=P_2=[0:0:1:0]$, $B=P_3=[0:0:0:1]$ (i.e. we have $L=\{x_0=x_1=0\}$) and also $\bP(\ker(H_f(A)))=P_0=[1:0:0:0]$ and $\bP(\ker(H_f(B)))=P_1=[0:1:0:0]$ (i.e. $\cV=\{x_2=x_3=0\}$). We can now write the Hessian matrix of $f$ evaluated at the points $P_2$ and $P_3$:
    $$H_f(P_2)=\begin{pmatrix}        0&0&0&0\\0&a_2&b_2&c_2\\0&b_2&d_2&e_2\\0&c_2&e_2&f_2
    \end{pmatrix}
    \qquad H_f(P_3)=\begin{pmatrix}        a_3&0&b_3&c_3\\0&0&0&0\\b_3&0&d_3&e_3\\c_3&0&e_3&f_3
    \end{pmatrix}$$
    Since $\cV$ is contained in both the quadrics defined by $H_f(P_2)$ and $H_f(P_3)$ we have for example that 
    $$0=P_1^tH_f(P_2)P_1=a_2, \qquad 0=P_0^tH_f(P_3)P_0=a_3.$$ 
    Moreover, denoting by $A^i$ the $i$-th column of a matrix $A$, with numeration from $0$ to $3$, by Schwartz, we have that
    $$H_f(P_2)\cdot P_3=H_f(P_2)^3=\left(\frac{\partial^3f}{\partial x_k\partial x_2\partial x_3}\right)_{k=0,\dots,3}=\left(\frac{\partial^3f}{\partial x_k\partial x_3\partial x_2}\right)_{k=0,\dots,3}=H_f(P_3)^2=H_f(P_3)\cdot P_2,$$
    i.e. we get
    $$b_3=0 \qquad c_2=0 \qquad e_2=d_3 \qquad f_2=e_3.$$
    To conclude, let us then write the Hessian matrix of $f$ evaluated at the general point of $L$, namely in $P_2+tP_3$:
    $$H_f(P_2+tP_3)=\begin{pmatrix}
        0&0&0&tc_3\\0&0&b_2&0\\0&b_2&d_2+te_2&e_2+tf_2\\tc_3&0&e_2+tf_2&f_2+tf_3
    \end{pmatrix}$$
    whose determinant is equal to $t^2b_2^2c_3^2$. Since we are assuming that $L$ is contained in $\cH_f$, such determinant has to be zero for every $t\in\bK$: then either $b_2=0$ or $c_3=0$. But this would imply that either $P_2$ or $P_3$, respectively, is a point in $\cD_2(f)$, against our assumptions. As claimed, the only possibility is that the intersection $L\cap\cV$ consists of a single point.
    \smallskip
    \ \\
    {\underline{Claim: If we assume that $\varphi$ is not constant, then $f$ can be written as $\tilde{f}$}}. Assuming that the map $\varphi$ is not constant, the only possibility is that the two lines $L$ and $\cV$ have a common point, namely $P$. Let us observe that since $P\in L\cap \cV$, we have by symmetry (and by the dominance of $\varphi$) that $\iota_f(P)\cap L\neq\emptyset$ and $\iota_f(P)\cap\cV\neq\emptyset$. From property \eqref{EQ:notdiagonal}, we have that $P\not\in\iota_f(P)$: from this we then get that $\iota_f(P)$ is a line (and so $P\in\cD_2(f)$). Let us fix also in this case a suitable system of coordinates as follows: 
    $$\{P_0\}=L\cap\iota_f(P), \quad \{P_1\}=\cV\cap\iota_f(P), \quad P=P_2.$$

    
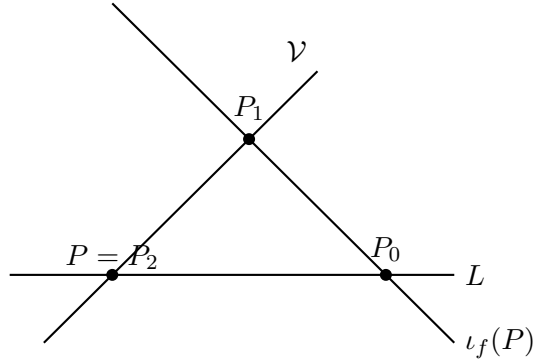
\begin{small}
  \begin{center}
\begin{tikzpicture}[scale=0.9]

\path[name path=L] (-2,0) -- (4,0);
\path[name path=V] (-2,-1) -- (2,3);
\path[name path=iota] (-1,4) -- (4,-1);

\path[name intersections={of=L and V, by=P=}];
\path[name intersections={of=L and iota, by=P0}];
\path[name intersections={of=V and iota, by=P1}];

\draw[thick] (-2.5,0) -- (4,0) node[right] {$L$};
\draw[thick] (-2,-1) -- (2,3) node[above left] {$\mathcal V$};
\draw[thick] (-1,4) -- (4,-1) node[right] {$\iota_f(P)$};

\fill (P) circle (2.5pt)
      node[above=6pt, inner sep=-7.5pt] {$P=P_2$};

\fill (P0) circle (2.5pt)
      node[above=4pt, inner sep=1pt] {$P_0$};

\fill (P1) circle (2.5pt)
      node[above=6pt, inner sep=1pt] {$P_1$};
\end{tikzpicture}
\end{center}
\end{small}

    Similarly, as before, we get
    $$H_f(P_0)=\begin{pmatrix}
        a_0&b_0&0&c_0\\b_0&d_0&0&e_0\\0&0&0&0\\c_0&e_0&0&f_0
    \end{pmatrix} \qquad H_f(P_2)=\begin{pmatrix}
        0&0&0&0\\0&0&0&0\\0&0&a_2&b_2\\0&0&b_2&c_2
    \end{pmatrix}. $$
    As before, since $\cV$ is contained in all the quadrics defined by the points of $L$, we get
    $$P_1^tH_f(P_0)P_1=0 \ \Longrightarrow \ d_0=0\qquad \mbox{and} \qquad P_2^tH_f(P_2)P_2=0 \ \Longrightarrow \ a_2=0.$$
    If we then, again, consider the matrix
    $$H_f(P_2+tP_0)=\begin{pmatrix}
        ta_0&tb_0&0&tc_0\\tb_0&0&0&te_0\\0&0&0&b_2\\tc_0&te_0&b_2&c_2+tf_0
    \end{pmatrix},$$ 
    we have that its determinant $t^2b_0^2b_2^2$ has to be identically zero (since $L\subset\cH_f$), i.e. either $b_0=0$ or $b_2=0$. Observe now that $b_2\neq0$, otherwise we would have that the point $P_2$ belongs to $\cD_1(f)$, and so $f$ would be of TS type, against our assumptions. Then, we must have $b_0=0$.\\
    Let us now translate the conditions obtained above in terms of the apolar ring $\cA_f$ associated of $f$. Recalling the identification $\bP^3\simeq\bP(A^1_f)$ and denoting by $\{y_0,\dots,y_3\}$ a basis of $(\cL_3)^1$ or, equivalently, of $A_f^1$, since $V(f)$ is not a cone (see Subsection \ref{SUBS:jacobianandapolar}), we have that:
    \begin{itemize}
        \item Having that $P_0, P_1\in\iota_f(P_2)$ means that $y_0y_2=0$ and $y_1y_2=0$ in $\cA_f$, i.e. in the expression of $f$ monomials of the form $x_0x_2\cdot\ell$ or $x_1x_2\cdot\ell$, where $\ell$ is any linear form, do not appear.
        \item $a_2=0$, $d_0=0$, and $b_0=0$ imply respectively that, in $\cA_f$, $y_2^3=0$, $y_1^2y_0=0$, and $y_0^2y_1=0$ i.e. in the expression of $f$ the monomials $x_2^3,x_1^2x_0,$ and $x_0^2x_1$ do not appear.        
    \end{itemize}
    Finally, since the fourth point $P_3$ is not subject to any condition, we can choose it in such way that it belongs to the cubic $V(f)$, i.e. in such a way that $y_3^3=0$ and so the term $x_3^3$ doesn't appear in the expression of $f$. To conclude, one can observe that with these constraints the form $f$ can be written in the same way as $\tilde{f}$ as in \eqref{EQ:specialsurface}, as claimed.\\
    Finally, it is easy to check that if $\tilde{f}$ is general in the linear system $(\star)$, then its hessian $\cH_{\tilde{f}}\subseteq\bP^3$ contains the line $L:\{x_3=x_1=0\}$ and there is not a point $Q$ such that $\iota_{\tilde{f}}(Q)=L$. Indeed, if so, this point $Q$ would be such that, by symmetry, $Q\in\iota_{\tilde{f}}(R)$ for any $R\in L$. But, by construction, $P_0,P_2\in L$ and $\iota_{\tilde{f}}(P_2)$ is a line not passing through $P_2$ (since $V(\tilde{f})$ is smooth), while $\iota_{\tilde{f}}(P_0)=P_2$.
\end{proof}

\medskip

Before moving to the higher dimension case, let us say that this kind of results seems to allow specific studies on the Hessian variety associated with a smooth cubic hypersurfaces. For example, as a consequence of the above Proposition \ref{PROP:surfacecase}, one can prove that the Hessian surface associated with a smooth cubic surface is never a cone. However, we do not move in this direction, since in Section \ref{SEC:hessiancones}, we will prove this result in a more general setting.\\
\subsection{Higher dimensions}
\label{SUBS:higher}
In this subsection, we will conclude the proof of Theorem \ref{THM:linearspaces}, by analyzing the case of smooth cubic hypersurfaces in $\bP^n$, for $n\geq4$. The proof, which is in some sense similar to that given for cubic surfaces in the previous subsection, will be divided into different steps.\\

Let then $f$ be a cubic form in $n+1$ variables not of TS type and defining a smooth hypersurface $X=V(f)$ in $\bP^n$, with $n\geq4$. Let us assume that there exists an $(n-2)$-plane $\Pi\simeq\bP^{n-2}$ contained in $\cH_f$. \\
Observe that the plane $\Pi$ is not contained in $\Sing(\cH_f)=\cD_{n-1}(f)$, otherwise $f$ would be of TS type by Theorem \ref{THM:charactTS}, against our assumptions. In other words, we have that the general point $P$ of $\Pi$ is smooth for $\cH_f$, i.e. not belonging $\cD_{n-1}(f)$ and so $\iota_f(P)$ consists of a single point. Let now $P_0,\dots,P_{n-2}$ be general $n-1$ points in general position in $\Pi$: let us define 
\begin{itemize}
    \item $T_i:=\iota_f(P_i)$;
    \item $\cQ_{P_i}$ the element in $J_f$ associated with $P_i$: it is the quadric defined by the symmetric matrix $H_f(P_i)$ with the point $T_i$ as vertex.
\end{itemize}
Let us now observe that, since $V(f)\subseteq\bP^n$ is smooth, the Jacobian $J_f$ is generated by a regular sequence of quadrics. Since the quadrics $\cQ_{P_i}$ are independent elements in $J_f$, we have that 
\begin{equation}
\label{EQ:baselocus}
\dim\left(\bigcap_{i=0}^{n-2}\cQ_{P_i}\right)=1.
\end{equation}
Let us now consider a rational map $\varphi:\Pi\dashrightarrow\cH_f$ sending a point $Q\in\Pi\setminus\cD_{n-1}(f)$ to its kernel $\iota_f(Q)$ and let us define the irreducible subvariety
$$\cV:=\overline{\mbox{Im}(\varphi)}.$$
Observe then that, by Bertini's Theorem, we have that $\cV\subseteq\cB:=\Bs(\Pi)$, the base locus of the linear system $\Pi$, i.e., in particular, $\dim(\cV)\leq1$ from Equation \eqref{EQ:baselocus}.
We then have that 
\begin{center}
    either $\cV$ is a point, or it is a curve.
\end{center}

In the first case, Theorem \ref{THM:linearspaces} is proved: indeed, we would have, by symmetry, that $\Pi\subseteq\iota_f(\cV)$, and clearly this is actually an equality, otherwise there would be a $\bP^{n-1}(\simeq\iota_f(\cV))$ as component of $\cH_f$, and so $f$ would be of TS type, against the assumptions.\\
The only remaining case is the one where $\cV$ is a curve: assuming this, let us prove that $f$ belongs to the linear system $(\star)$ (see \eqref{EQ:linearsystem}).

\noindent
Let $A\in\cV$ be a general point: there exists a subvariety $Y_A$ of $\Pi$ of dimension $n-3$, such that for $y\in Y_A$ general we have that $\varphi(y)=A$. In particular, we have that, by symmetry, $Y_A=\iota_f(A)\cap\Pi$, i.e. $Y_A\simeq\bP^{n-3}$.\\
Let us now consider another general point $A_t\in\cV$: in $\Pi$ we have then that $Y_{A}\cap Y_{A_t}=:\Lambda_t\simeq\bP^{n-4}$. Observe now that for any $z\in\Lambda_t$, we would have $\iota_f(z)\supset\bP(\left\langle A,A_t\right\rangle)$, i.e. $\Lambda_t\subset\cD_{n-1}(f)$. Hence, if such an intersection $\Lambda_t$ varies with the point $A_t$ we would have, since $A$ is general, that the entire $\Pi$ is contained in the singular locus $\cD_{n-1}(f)$ of $\cH_f$ against our assumptions. Then there exists an $(n-4)$-plane $\Lambda$ which is contained in $Y_A$ for every $A\in\cV$. Let us now prove the following:

\begin{lemma}
\label{LEM:Vline}
    $\cV$ is a projective line.
\end{lemma}

\begin{proof}
    Let us first observe that by construction for any $z\in\Lambda$, we have $\iota_f(z)\supseteq\cV$. This means that $\Lambda\subseteq\cD_{n-1}(f)$. \\
    Clearly, if the general point $z\in\Lambda$ lives in $\cD_{n-1}(f)\setminus\cD_{n-2}(f)$, then $\bP^1\simeq\iota_f(z)\supset\cV$: we get the claim. On the other hand, it is not possible that the $\Lambda$ is contained in $\cD_{n-3}(f)$: since $\Lambda\simeq\bP^{n-4}$ we would have, against our hypotheses, that $f$ is of TS type by Theorem \ref{THM:charactTS}.\\
    Hence, the only remaining case to be analyzed is that where the general point $z$ of $\Lambda$ lives in $\cD_{n-2}(f)\setminus\cD_{n-3}(f)$, i.e. we have that for $[z]\in\Lambda$ general $\iota_f([z])\simeq\bP^2\supseteq\cV$. Observe that if $\Lambda$ has positive dimension, we can assume that $\iota_f([z])$ is constant with $[z]$ varying in $\Lambda$, otherwise we would have the thesis, since $\cV$ is contained in distinct $2$-planes. However, we have that $\cV$ is a plane curve and let us denote by $\Sigma$ a $2$-plane containing it. Recall now that $\cV$ consists of the vertices of the quadrics parametrized by the linear system $\Pi$ and that, consequently, it is an irreducible component of $\cB=\Bs(\Pi)$, which has dimension $1$, as shown above. Recall also that we define $P_0,\dots,P_{n-2}$ to be $n-1$ points in $\Pi$ in general position, and we set $T_i=\iota_f(P_i)$. These points $T_i$ are then in $\cV$: for any $i=0,\dots,n-2$. Let us consider the joint variety $J_i=J(T_i,\cV)$: if we assume by contradiction that $\cV\not\simeq\bP^1$, we can easily see that $J_i\simeq \Sigma$ is contained in $\cQ_{P_i}$ (the quadric associated with the point $P_i$). But the same holds for any $i=0,\dots,n-2$: the $2$-plane $\Sigma$ is then contained in all the quadrics $\cQ_{P_i}$ and so in the base locus $\cB$. Since it has dimension $1$, we get then a contradiction: $\cV\simeq\bP^1$, as claimed.
\end{proof}

We will now study the mutual position between the line $\cV$ and the $(n-2)$-plane $\Pi$. In particular, we prove the following:
\begin{lemma}
    The line $\cV$ and the $(n-2)$-plane $\Pi$ have a single common point, i.e.
    $$\cV\cap\Pi=\{R\}.$$
\end{lemma}

\begin{proof}
    We have clearly three possibility: either the line $\cV$ is contained in $\Pi$, or they are disjoint, or they meet in a point. To prove the Lemma, let us rule out the first two cases. \\
    Let us start by assuming that $\cV\subset\Pi$. With the notation introduced above, one can observe that $\cV$ intersects $Y_A$ in a single point: indeed, if we assume that $\cV$ is contained in $Y_A$ for some $A\in\cV$, we would have by construction that $A\in\iota_f(A)$, i.e. $A$ would be a singular point for $V(f)$, again our hypothesis of smoothness (see property \eqref{EQ:notdiagonal}). This means that for any $A\in\cV$, there exists a point $A'\in\cV\cap Y_A$ such that $A'\in\iota_f(A)$. We can then consider the incidence variety $$\Gamma_{f_{|\cV\times\cV}}=\{(A,A')\in\cV\times\cV \ | \ A'\in\iota_f(\cV)\}\subset\cV\times\cV\simeq\bP^1\times\bP^1.$$
    This incidence correspondence has dimension $1$ by construction and so the intersection with $\Delta_{\bP^1}$ is not trivial, giving a singular point for $V(f)$: again, a contradiction.\\
    Let us now assume that $\cV$ and $\Pi$ are disjoint and as we did in the case of surfaces in the previous Subsection, let us now fix a suitable system of coordinates.

    
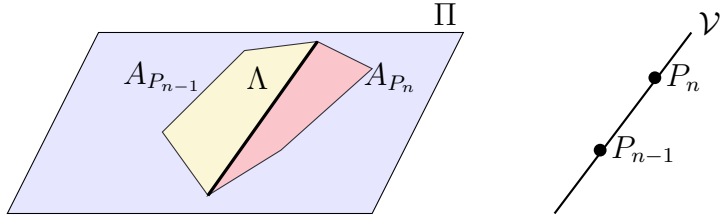
\begin{center}
    \begin{tikzpicture}[scale=1.2]

    \filldraw[fill=blue!10, draw=black]
    (-1,-1) -- (3,-1) -- (4,1) -- (0,1) -- cycle;
    \node at (3.8,1.2) {$\Pi$};

    \filldraw[fill=red!25, draw=black, opacity=0.8]
    (1.2,-0.8) -- (2,-0.3) -- (3,0.6) -- (2.4,0.9) -- cycle;
    \node at (3.2,0.5) {$A_{P_n}$};

    \filldraw[fill=yellow!25, draw=black, opacity=0.8]
    (1.2,-0.8) -- (0.7,-0.1) -- (1.6,0.8) -- (2.4,0.9) -- cycle;
    \node at (0.7,0.5) {$A_{P_{n-1}}$};

    \draw[very thick] (1.2,-0.8) -- (2.4,0.9);
    \node[right] at (1.5,0.5) {$\Lambda$};
    
    \draw[thick] (5,-1) -- (6.5,1);
    \node at (6.7,1.1) {$\mathcal V$};

    \fill (5.5,-0.3) circle (2pt);
    \node[right] at (5.5,-0.3) {$P_{n-1}$};

    \fill (6.1,0.5) circle (2pt);
    \node[right] at (6.1,0.5) {$P_n$};

    \end{tikzpicture}
    \end{center}

    As in the picture, let us take the coordinate points $P_n$ and $P_{n-1}$ as two general distinct points on $\cV$. Moreover, let us consider $P_0,\dots,P_{n-4}$ general points in general position in $\Lambda$; and finally $P_{n-3}\in A_{P_{n-1}}$ and $P_{n-2}\in A_{P_n}$.
    From this configuration, we then can write the first vanishing relations in the apolar ring $\cA_f$:
    \begin{equation}
        \label{EQ:vanishing1}
        y_n\cdot y_i=0 \quad \forall i\in\{0,\dots,n-2\}\setminus\{n-3\} \qquad \mbox{and} \qquad y_{n-1}\cdot y_i=0 \quad \forall i\in\{0,\dots,n-3\}.
    \end{equation}
    Observe now that for any $t\in\bK$, there exists $\mu(t)$ (with $\mu(0)=0$, but clearly $\mu(t)\not\equiv0$), such that
    $$(y_{n-3}+ty_{n-2})\cdot(y_{n-1}+\mu(t)y_n)=0, \quad \mbox{i.e.} \quad \mu(t)y_{n-3}y_n+t y_{n-2}y_{n-1}=0.$$
    Multiplying this by $y_{n-3}$, $y_{n-2}$, $y_{n-1}$, and $y_n$, we get respectively
    \begin{equation}
        \label{EQ:vanishing2}
        y_{n-3}^2y_n=0, \qquad y_{n-2}^2y_{n-1}=0, \qquad y_{n-2}y_{n-1}^2=0, \qquad y_{n-3}y_n^2=0.
    \end{equation}
    Let us now focus on the product $y_{n-2}y_{n-1}\in A_f^2$: by conditions \eqref{EQ:vanishing1} and \eqref{EQ:vanishing2} we have that 
    $$y_{n-2}y_{n-1}y_i=0 \quad \forall i=0,\dots,n.$$
    Since the apolar ring $\cA_f$ is a Gorenstein algebra, this means that $y_{n-2}y_{n-1}=0$, and in the same way one gets $y_ny_{n-3}=0$. In other words, we have
    $$\iota_f(P_{n-1})\simeq\bP^{n-2} \quad \mbox{and} \quad \iota_f(P_n)\simeq\bP^{n-2} \qquad \mbox{i.e.} \quad P_{n-1},P_n\in\cD_2(f)$$

Since they have been chosen generically, we have that $\cV\subseteq\cD_2(f)$, but since $\cV\simeq\bP^1$, we would have that $f$ is of TS type by Theorem \ref{THM:charactTS}, against our assumptions. The only possibility is that $\cV$ and $\Pi$ meet each other in a single point, as claimed.
\end{proof}

Exploiting now the arising configuration, as done in the case of surfaces in the previous subsection, one can explicitly describe the expression of the associated cubic form $f$.\\
Let us do it in detail in the case where $n=4$, i.e. for cubic threefolds:

\begin{configuration}
\label{CONF:3-folds}
    As proved above, we necessarily have that the line $\cV$ and the $2$-plane $\Pi$ intersect each other in a point, denoted by $R$. \\
    Let us translate in this case the configuration described above: here $\Lambda$ is just a point denoted by $P$ and the $2$-plane $\Pi$ is covered by the linear subvarieties $Y_A$, with $A\in\cV$, which are lines passing through $P$ and such that for any point $x\in Y_A$ we have $\iota_f(x)\ni A$ (with $A=\iota_f(x)$, for $x$ general in $Y_A$). Let us observe that clearly $R\neq P$, otherwise it would be a singular point for $V(f)$. By the dominance of the map $\varphi$ we have that there exists a line passing through $P$ and contained in $\Pi\cap\iota_f(R)$; moreover, there exists a point $T\in\cV$ such that the line $\bP(\langle P,R\rangle)\subseteq\iota_f(T)$. By symmetry, this means that $R\in\cD_2(f)$ and that $\iota_f(R)\simeq\bP^2$ intersecting $\Pi$ along a line and cutting $\cV$ in the point $T$. \\
    Let us now fix a system of coordinates in the following way:
    $$P=P_0=[1:0:0:0:0], \quad P_1=[0:1:0:0:0]\in\iota_f(R)\cap\Pi \mbox{ general},$$
    $$T=P_2=[0:0:1:0:0], \quad R=P_3=[0:0:0:1:0].$$

    
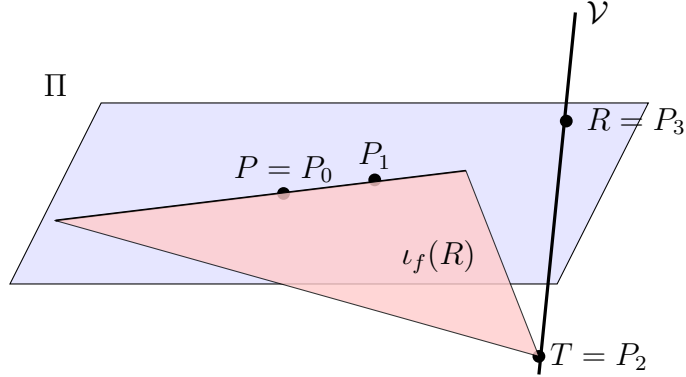
\begin{center}
    
    \begin{tikzpicture}[scale=1.2]

    \filldraw[fill=blue!10, draw=black] 
    (-3,-1) -- (3,-1) -- (4,1) -- (-2,1) -- cycle;
    \node at (-2.5,1.2) {$\Pi$};

    \fill (0,0) circle (2pt);
    \node[above] at (0,0) {$P=P_0$};

    \draw[thick] (-2.5,-0.3) -- (2.0,0.25);

    \fill (1.0,0.15) circle (2pt);
    \node[above] at (1.0,0.15) {$P_1$};

    \draw[very thick] (3.2,2) -- (2.8,-2);
    \node[right] at (3.2,2) {$\mathcal V$};

    \fill (3.1,0.8) circle (2pt);
    \node[right] at (3.2,0.8) {$R=P_3$};

    \fill (2.8,-1.8) circle (2pt);
    \node[right] at (2.8,-1.8) {$T=P_2$};

    \filldraw[fill=red!20, draw=black, opacity=0.8]
    (-2.5,-0.3) -- (2.0,0.25) -- (2.8,-1.8) -- cycle;
    \node at (1.7,-0.7) {$\iota_f(R)$};

    \end{tikzpicture}
    \end{center}

    As in the case of surfaces, we can then interpret such a construction via the following vanishing conditions in the apolar ring $\cA_f$:
    \begin{equation}
    \label{EQ:firstconditionthreefolds}
    y_0y_2=0, \quad y_0y_3=0, \quad y_1y_3=0, \quad y_2y_3=0.
    \end{equation}

    Furthermore, as above, we have that for every $t$, there exists $\mu(t)$ (with $\mu(0)=0$, but clearly $\mu(t)\not\equiv0$), such that in the apolar ring we have
    $$(y_1+ty_3)\cdot(y_3+\mu(t)y_2)\equiv0 \rightarrow \mu(t)y_1y_2+ty_3^2\equiv0.$$
    Multiplying by $y_3$, we get that 
    \begin{equation}
    \label{EQ:secondconditionsthreefolds}
        y_3^3=0,
    \end{equation}
    i.e. that the point $R=P_3$ belongs to the cubic $X=V(f)$. Moreover, multiplying by $y_1$ or $y_2$, we respectively get
    \begin{equation}
    \label{EQ:thirdconditionsthreefolds}
        y_1^2y_2=0 \quad \mbox{and} \quad y_1y_2^2=0.
    \end{equation}

    By recollecting the conditions in Equations \eqref{EQ:firstconditionthreefolds}, \eqref{EQ:secondconditionsthreefolds}, and \eqref{EQ:thirdconditionsthreefolds}, one obtains $19$ independent conditions of degree three in $\cA_f$; once these are imposed, and observing that we can still choose the fifth point $P_4$ in such a way it belongs to the cubic $V(f)$, we can write the cubic $f$ as a general element of the linear system $(\star)$, as in the statement of Theorem \ref{THM:linearspaces}.
\end{configuration}
\medskip
\ \\
Let us just say a few words for $n\geq5$: in the configuration arising above, let us denote by $P_{n-1}$ the intersection point between the $(n-2)$-plane and the line $\cV$. Let us moreover set $P_0,\dots,P_{n-4}$, general points in $\Lambda$ in general position, $P_{n-3}$ a general point in $(\iota_f(P_{n-1})\cap\Pi)\setminus\Lambda$ and $P_{n-2}$ the intersection point between $\cV$ and $\iota_f(P_{n-1})$.
In the apolar ring $\cA_f$, we have then the following immediate relations:
\begin{equation}
    \label{EQ:firstconditionshighdimension}
     y_{n-1}y_i=0 \quad  \forall i = 0,\dots,n-2 \quad \mbox{and} \quad y_{n-2}y_i=0 \quad \forall i=0,\dots,n-4.
\end{equation}
Moreover, as done above, we know that for any $t\in\bK$, there exists $\mu(t)$ (with $\mu(0)=0$) such that 
$$(y_{n-3}+ty_{n-1})\cdot(y_{n-1}+\mu(t)y_{n-2})=0 \quad \mbox{from which we get} \quad ty_{n-1}^2+\mu(t)y_{n-3}y_{n-2}\equiv0.$$
Multiplying then by $y_{n-1}$, $y_{n-2}$, and $y_{n-3}$, we get respectively
\begin{equation}
\label{EQ:secondconditionshighdimension}
y_{n-1}^3=0, \qquad y_{n-2}^2y_{n-3}=0, \qquad \mbox{and} \qquad y_{n-2}y_{n-3}^2=0. 
\end{equation}
Finally, since the last point $P_n$ can be chosen on the cubic hypersurface $V(f)$, we have also in $\cA_f$ that $y_n^3=0$. Considering this condition, together with those in Equations \eqref{EQ:firstconditionshighdimension} and \eqref{EQ:secondconditionshighdimension}, we find that $f$ has to belong to the linear system $(\star)$ in \eqref{EQ:linearsystem}, as claimed.\\
To conclude, recall that by Remark \ref{REM:familysmooth} the general cubic in the system $(\star)$ is indeed smooth. 
\begin{remark}
    Let us prove here that the general cubic $\tilde{f}$ in the linear system $(\star)$ (Equation \eqref{EQ:linearsystem}) satisfy the desired conditions, namely
    \begin{enumerate}[(a)]
        \item the Hessian hypersurface $\cH_{\tilde{f}}$ contains an $(n-2)$-plane $\Pi$;
        \item there exists no point $P\in\cH_{\tilde{f}}$ such that $\Pi=\iota_{\tilde{f}}(P)$.
    \end{enumerate}
    Indeed, for $(a)$, one can easily see that taking the Hessian matrix $H_{\tilde{f}}$ of a general form $\tilde{f}$ in the system $(\star)$ and the $(n-2)$-plane $\Pi$, defined as in the construction above by $\Pi:=\{x_{n-2}=x_n=0\}$, one has that the matrix $H_{\tilde{f}|\Pi}$ has two dependent rows, i.e. $\Pi\subseteq\cH_{\tilde{f}}$.\\
    For $(b)$, one can see that by construction the point $P_{n-1}$ is in $\Pi$ and $\iota_{\tilde{f}}(P_{n-1})\simeq\bP^{n-2}\not\ni P_{n-1}$. On the other hand, if $R$ is a general point on $\Pi$, we know by construction that it is smooth for $\cH_{\tilde{f}}$ and $\iota_{\tilde{f}}(R)$ consists of a single point in $\cV$, not belonging to $\iota_{\tilde{f}}(P_{n-1})$. This means that $\iota_{\tilde{f}}(R)\cap\iota_{\tilde{f}}(P_{n-1})=\emptyset$, i.e., by symmetry, the $(n-2)$-plane $\Pi$ doesn't coincide with $\iota_{\tilde{f}}(P)$ for any point $P\in\cH_{\tilde{f}}$, as claimed.
\end{remark}

Theorem \ref{THM:linearspaces} is then completely proved.

\begin{remark}
    \label{REM:specialthreefold}
    Since it will be used in the following section, let us explicitly write the general element of the linear system $(\star)$, for $n=4$, hence defining a smooth cubic threefold in $\bP^4$:
    \begin{equation}
    \label{EQ:special3fold}
    \tilde{f}(x_0,x_1,x_2,x_3,x_4)=\alpha_0x_0^3+\alpha_1x_1^3+\alpha_2x_2^3+\alpha_3x_0^2x_1+\alpha_4x_1^2x_0+
     \end{equation}
     $$x_4(\beta_0x_0^2+\beta_1x_1^2+\beta_2x_2^2+\beta_3x_3^2+\beta_4x_0x_1+\beta_5x_1x_2)+x_4^2(\gamma_0x_0+\gamma_1x_1+\gamma_2x_2+\gamma_3x_3).$$
\end{remark}

\bigskip

Let us now conclude this section just saying a couple of words for the case of cubic forms of Thom-Sebastiani type, in the case of smooth cubic threefolds and fourfolds.

\begin{remark}
    Let $f$ be a cyclic polynomial defining a smooth threefold in $\bP^4$. It can be written as
    $$f=x_0^3+f_1(x_1,\dots,x_4)$$
    Let us then assume there exists a $2$-plane $\Pi\simeq\bP^2$ in $\cH_f$ and not entirely contained in $\Sing(\cH_f)$ (this case has been treated in Lemma \ref{LEM:TSSing}). With the same notation as that in Lemma \ref{LEM:TSSing}, we can write 
    $$\cH_f=V(x_0)\cup\tilde{\cH}_{f_1},$$
    where $\tilde{\cH}_{f_1}$ is the cone of $\cH_{f_1}$ over the point $P_0=[1:0:0:0:0]$. If $\Pi\subset V(x_0)=\iota_f(P_0)$, then we have that $\Pi$ is contained in the kernel of some point of the Hessian variety. Assume now that $\Pi$ is contained in $\tilde{\cH}_{f_1}$.\\
    If it is contained in $\cH_{f_1}$, then $f_1$ is itself of TS type by Theorem \ref{THM:TS}. Then we can write $f=x_0^3+x_1^3+f_2(x_2,x_3,x_4)$. Hence, there are two cases: either $f$ is the Fermat polynomial or $\Pi\subset V(x_0)\cup V(x_1)$. In both cases, we have that there exists a point $P\in\cH_f$ such that $\Pi\subseteq\iota_f(P)$ (for the former case, see Remark \ref{REM:Fermat}).\\
    The remaining possibility is that $\Pi$ is the joint variety $J=J(P_0,L)$, with $L$ being a line in $\cH_{f_1}$. See now Proposition \ref{PROP:surfacecase}: if there exists a point $P$ such that $L\subseteq\iota_{f_1}(P)$, then $\Pi\subseteq\iota_f(P)$, and we get the claim. \\
    Summarizing, if $\Pi$ is a $2$-plane contained in $\cH_f\setminus\Sing(\cH_f)$, there always exists a point $P\in\cH_f$ such that $\Pi\subseteq\iota_f(P)$, except for the case where $f$ can be written as 
    $$f(x_0,\dots,x_4)=x_0^3+\tilde{f}(x_1,\dots,x_4),$$
    where $\tilde{f}$ is as in Equation \eqref{EQ:specialsurface}.\\
    One can do the same reasoning for the case of cyclic cubic fourfolds: the only case where our resul does not hold is the one where the cubic form $f$ can be written as
    $$f(x_0,\dots,x_5)=x_0^3+\tilde{f}_1(x_1,\dots,x_5)$$
    or
    $$f(x_0,\dots,x_5)=x_0^3+x_1^3+\tilde{f}_2(x_2,\dots,x_5),$$
    where $\tilde{f}_1$ and $\tilde{f}_2$ are respectively those of Equations \eqref{EQ:special3fold} or \eqref{EQ:specialsurface}.
\end{remark}

Let us now consider the only remaining case:
\begin{remark}
    Let us now assume that $X=V(f)\subset\bP^5$ is a smooth cubic fourfold, defined by a polynomial $f$ of TS type but not cyclic. We can then write 
    $$f(x_0,\dots,x_5)=f_1(x_0,x_1,x_2)+f_2(x_3,x_4,x_5).$$
    One can refer to \cite{BFP3}[Example 2.8] for a detailed description. Since we are assuming that $f$ is not cyclic, then by Theorem \ref{THM:TS} we have that $\cH_{f_1}$ and $\cH_{f_2}$ are smooth cubic curves. From the analysis of $\cH_f$ it is then clear that there are no possibilities for a $3$-plane to be contained in $\Sing(\cH_f)$. Observe that $\cH_f$ consists of two components defined as the cones of $\cH_{f_i}$ over the $2$-plane $W_j$ for $i\neq j$, where $W_1=\{x_0=x_1=x_2=0\}$ and $W_2=\{x_3=x_4=x_5=0\}$: we have infinitely many $3$-planes in $\cH_f$, each of which is a joint variety $J(P_i,W_j)$, where $P_i\in\cH_{f_i}$. But then, by construction there is a point $T_i\in\cH_{f_i}$ such that $\iota_{f_i}(T_i)=P_i$, and so $\iota_f(T_i)=J(P_i,W_j)$. Hence, we have the claim: for any $3$-plane $\Pi$ there exists a point $P$ such that $\Pi\subseteq\iota_f(P)$.
\end{remark}

\section{An application: the reconstruction of some cubic threefolds from their hessians}
\label{SEC:reconstruction}

In this section, we give an application of the results presented in Section \ref{SEC:linearspaces}. For this, let us consider the hessian map
$$h_{d,n}:\bP(S^d_n)\dashrightarrow\bP(S_n^{(d-2)(n+1)}),$$
mapping a polynomial $f$ of degree $d$ to its associated hessian polynomial $h_f$. Observe that, because of the existence of polynomials in the Gordan-Noether locus, where the hessian polynomial vanishes identically, such a hessian map is not a morphism. Ciliberto and Ottaviani in \cite{CO} conjectured the following:
\begin{conjecture}[Ciliberto-Ottaviani]
\label{CONJ:CO}
    The hessian map $h_{d,n}$ is birational onto its image for every $d\geq3$ and for every $n\geq2$, except for the case where $(d,n)=(3,2)$.
\end{conjecture}
Again in \cite{CO}, the authors showed that the hessian map is generically finite, and proved the conjecture for cubic surfaces in $\bP^3$.\\ 
Here, we focus on the case of cubic threefolds, i.e. on the hessian map
$$h_{3,4}:\bP(S^3_4)\dashrightarrow\bP(S^5_4),$$
sending a cubic form in $5$ variables which is not of Gordan-Noether type to its hessian polynomial, which is a quintic form in $5$ variables.\\
We recall the following:
\begin{definition}
\label{DEF:WR}
    We say that a polynomial $f\in\bP(S^d_n)$ has Waring Rank $k$, if, up to a linear change of coordinates there exists $k$ linear forms $L_0,\dots,L_{k-1}$ such that $f=\sum_{i=0}^{k-1}L_i^d$.
\end{definition}

In \cite{CO}, Ciliberto and Ottaviani studied the locus of forms of Waring Rank $n+2$ in $\bP(S^d_n)$ and proved the following:
\begin{theorem}[\cite{CO},Theorem 6.5]
    \label{THM:WRn+2}
    The restriction of the map $h_{d,n}$ to the locus of forms of Waring Rank $n+2$ is birational onto its image. 
\end{theorem}
Here, in the case of cubic threefolds, i.e. with $n=4$, we denote by $\cW_6\subset\bP(S^3_4)$ the set of cubic forms in $5$ variables of waring rank $6$.
Recall also that we denote by $\cU_4\subset\bP(S^3_4)$, the locus of cubic forms defining a smooth cubic threefold in $\bP^4$. In this section, we propose a sort of generalization of the above Theorem \ref{THM:WRn+2} in the case of cubic threefolds, and, so, in some sense, an evidence to the Ciliberto-Ottaviani conjecture \ref{CONJ:CO}. By saying this, we mean that we show that, given a general $[f]\in\cW_6$, the only element in $\cU_4$ having $h_f$ as associated hessian polynomial is $f$ itself (see Theorem  \ref{THM:CO}).\\

Let $[f]\in\cW_6$ general (in particular, $f$ is not of Waring Rank $k\leq5$): we can write
\begin{equation}
\label{EQ:generalWR}
    f=\sum_{i=0}^5L_i^3,
\end{equation}
or, without loss of generality, we can also consider a kind of {\it{normal form}}

\begin{equation}
    \label{EQ:WRnormalform}
    f=x_0^3+x_1^3+x_2^3+x_3^3+x_4^3+(a_0x_0+a_1x_1+a_2x_2+a_3x_3+a_4x_4)^3,
\end{equation}
where we have $L_i=x_i$ for $i=0,\dots,4$ and $L=L_5=\sum_{i=0}^4a_ix_i$.
For such an $f$, one can compute its Hessian matrix $H_f$ and its hessian polynomial $h_f$, respectively:
\begin{equation}
\label{EQ:HessianMatrixOfWaringRankn+2}
H_f=\Delta+L\cdot(a_ia_j)_{i,j=0}^n=\Delta+L a\cdot a^T,
\end{equation}
where $\Delta$ is the diagonal matrix $\diag(x_0,x_1,\dots,x_4)$ and
\begin{equation}
\label{EQ:HessianPolynomialOfWaringRankn+2}
h_f=\prod_{i=0}^n x_i+L\sum_{i=0}^n a_i^2\hat{x}_i,    
\end{equation}
where $\hat{x}_i=\prod_{j\neq i}x_j$.\\
First of all, one can easily observe that the general element $[f]\in\cW_6$ is in $\cU_4$, i.e. it defines a smooth cubic threefold in $\bP^4$.



Observe that if $[f]\in\cW_6$ is general, we can assume the following:
\begin{equation}
    \label{EQ:descriptionoff}
    f \mbox{ of Waring Rank exactly 6,} \quad [f]\in\cU_4 \quad \mbox{and} \quad f \mbox{ not of TS type.}
\end{equation}
In particular, in this case, observe that if $f$ is written as in \eqref{EQ:WRnormalform}, then any coefficient $a_i$ appearing in the linear form $L$ is different from zero (otherwise $f$ would be of TS type).\\
Let us now give a description of the hessian variety $h_f$ associated with $f$ (see also \cite{CO}[Prop. 6.4]):
\begin{lemma}
    Let $[f]$ be a general cubic form as in \eqref{EQ:descriptionoff}, written in the form \eqref{EQ:generalWR}. The hessian threefold $\cH_f$ associated with $f$ is singular along the twenty lines defined as the zero of three of the linear forms $L_i$'s, i.e.
    $$\Sing(\cH_f)=\bigcup_{i,j,k \mbox{ distinct in } \{0,\dots,5\}}V(L_i,L_j,L_k).$$
\end{lemma}
One can observe that these twenty lines intersect each other in $15$ distinct points, of the form $V(L_i,L_j,L_k,L_m)$ for $i,j,k,m\in\{0,\dots,5\}$ distinct. Let us now described how the correspondence defined by $\iota_f$ acts on these linear spaces. Let us start with the following:
\begin{definition}
    \label{DEF:constantlines}
    Given $[f]\in\bP(S^3_4)$ and a line $\ell\subset\cD_{3}(f)$, with $\ell\not\subseteq\cD_{2}(f)$, we say that $\ell$ has constant kernel (with respect to $\iota_f$) if there exists another line $r$ such that 
    $$\iota_f(P)\supseteq r \quad \forall P\in\ell \quad \mbox{and equality holds for P general}.$$
    In this case, we write $\iota_f(\ell)=r$.\\
    Similarly, if for a linear space $\Pi$ one has that for $P\in\Pi$ general $\iota_f(P)$ is constant, we simply write $\iota_f(\Pi)$ to mean $\iota_f(P)$, for $P\in\Pi$ general.
\end{definition}

Setting $\{0,1,2,3,4,5\}=\{i,j,k,l,m,n\}$, let us denote 
$$P_{ijkl}=V(L_i,L_j,L_k,L_l) \qquad \ell_{ijk}=V(L_i,L_j,L_k) \qquad \Pi_{ij}=V(L_i,L_j),$$
which are respectively points, lines and $2$-planes in $\bP^4$.
From the expression \eqref{EQ:HessianMatrixOfWaringRankn+2} of the Hessian matrix, one can easily prove the following:

\begin{lemma}
    \label{LEM:iotaf}
    Let $[f]$ as in \eqref{EQ:descriptionoff} written as in the expression \eqref{EQ:generalWR}. Then we have:
    \begin{itemize}
        \item each of the $20$ lines in $\Sing(\cH_f)$ has constant kernel, and in particular
        $$\iota_f(\ell_{ijk})=\ell_{lmn},$$
        \item $\cD_2(f)$ consists exactly of the $15$ points $P_{ijkl},$
        \item each of the $15$ points in $\cD_2(f)$ is such that
        $$\iota_f(P_{ijkl})=\Pi_{mn}\simeq\bP^2 \qquad \mbox{and also} \qquad \iota_f(\Pi_{mn})=P_{ijkl}.$$
    \end{itemize}
\end{lemma}

We have then a rich configuration in $\cH_f$ given by these linear spaces. In particular:
\begin{lemma}
    \label{LEM:configuration}
    In the above situation, we have:
    \begin{itemize}
    \item Each of the $15$ planes $\Pi_{ij}$ contains exactly $4$ lines in $\Sing(\cH_f)$ and $6$ points in $\cD_2(f)$. Moreover, each of the $20$ lines in $\Sing(\cH_f)$ passes through $3$ points in $\cD_2(f)$.
    \item Each point in $\cD_2(f)$ belongs to four distinct lines in $\Sing(\cH_f)$ and each of these lines lies on $3$ of the $2$-planes above.
\end{itemize}
\end{lemma}

This rich geometry allows us to prove the main theorem of this section, namely Theorem C from the Introduction:

\begin{theorem}
    \label{THM:CO}
    Let $[f]\in\cW_6$ general, then considering the restriction $h_{{3,4}_{|\cU_4}}$ of the hessian map to the smooth locus $\cU_4\subseteq\bP(S^3_4)$, we have
    $$h_{{3,4}_{|\cU_4}}^{-1}(h_{{3,4}_{|\cU_4}}(f))=\{[f]\}.$$
    In other words, if $[g]$ is a cubic form in $\cU_4$ such that $\cH_g=\cH_f$, then $[f]=[g]$.
\end{theorem}

To prove the Theorem, let $[f]$ be a general element in $\cW_6$ as in \eqref{EQ:descriptionoff} and let us assume that $[g]\in\cU_4$ is a cubic form such that $\cH:=\cH_f=\cH_g$.
\begin{remark}
    Obviously, from the equality $\cH_f=\cH_g$ one has that also the singular loci coincide, i.e. $\Sing(\cH_f)=\Sing(\cH_g)$. Since both $[f]$ and $[g]$ are in $\cU_4$, by Theorem \ref{THM:preliminaries}, we have
    $$\cD_3(f)=\cD_3(g).$$
    Observe that, of course, $\Sing(\cD_3(f))=\Sing(\cD_3(g))$, but a priori it is not true that the loci $\cD_2(f)$ and $\cD_2(g)$ coincide. Indeed, in this case, we know that $\cD_2(f)=\Sing(\cD_3(f))$ (see Lemma \ref{LEM:iotaf}), but a priori we just have $\cD_2(g)\subseteq\Sing(\cD_3(g)$ (for this last fact, one can see for example \cite{BFP2}).
\end{remark}

\begin{remark}
    \label{REM:gnotTS}
    Since $\Sing(\cH_f)=\Sing(\cH_g)$, which has dimension $1$, we get that $g$ is not of Thom-Sebastiani type (by Theorem \ref{THM:TS}). Hence, by Proposition \ref{PROP:Grenoble} and Theorem \ref{THM:Sigmaplus} we have respectively
    $$\cD_1(g)=\emptyset \qquad \mbox{and} \qquad \dim(\cD_2(g))\leq0.$$
\end{remark}
\medskip
\ \\
Let us present here the steps of the proof. In Subsection \ref{SUBS:constantlines} we will prove that the lines contained in $\cD_3(g)(=\cD_3(f))$ have constant kernel with respect to $\iota_g$ in the sense of Definition \ref{DEF:constantlines}. In Subsection \ref{SUBS:coincideiota} we will show that the correspondences $\iota_f$ and $\iota_g$ act in the same way on $\Sing(\cH)$. Finally, in Subsection \ref{SUBS:againWR} we will show that $g$ itself lives in $\cW_6$. At this point, one can conclude by using the above Theorem \ref{THM:WRn+2}.

\subsection{Step 1}
\label{SUBS:constantlines}
\ \\
From Section \ref{SEC:linearspaces}, we know that given a smooth cubic threefold $V(F)$ not of TS type then either the cubic form $F$ can be written as in Equation \eqref{EQ:special3fold} or for any $2$-plane $\Pi$ contained in $\cH_F$ there exists a point $P\in\cD_2(F)$ such that $\Pi=\iota_F(P)$. Let us show that in our situation, in the hypothesis of Theorem \ref{THM:CO}, the cubic form $g$ can not be written as in the expression \eqref{EQ:special3fold}. Let us start with the following:
\begin{lemma}
\label{LEM:bruttaftilde}
    Let $\tilde{f}$ be as in Equation \eqref{EQ:special3fold}, and assume it is not of Thom-Sebastiani type. Then there exists a $2$-plane in $\cH_{\tilde{f}}$ which intersects $\cD_3(f)$ in $2$ lines.
\end{lemma}
\begin{proof}
    First of all, let us recall that from the construction of $\tilde{f}$ in Subsection \ref{SUBS:higher} (see Configuration \ref{CONF:3-folds}), we know that in $\cH_{\tilde{f}}$ there exists the $2$-plane $\Pi$ passing through the coordinate points $P_0,P_1,P_3$, i.e. $\Pi=\{x_2=x_4=0\}$. Let us now write the Hessian matrix of $\tilde{f}$ evaluated in this plane $\Pi$:
   
    $$H_{{\tilde{f}}|\Pi}=\begin{pmatrix}
        6\alpha_0x_0+2\alpha_3x_1&2\alpha_3x_0+2\alpha_4x_1&0&0&2\beta_0x_0+\beta_4x_1\\2\alpha_3x_0+2\alpha_4x_1&6\alpha_1x_1+2\alpha_4x_0&0&0&2\beta_1x_1+\beta_4x_0\\0&0&0&0&\beta_5x_1\\0&0&0&0&2\beta_3x_3\\2\beta_0x_0+\beta_4x_1&2\beta_1x_1+\beta_4x_0&\beta_5x_1&2\beta_3x_3&2(\gamma_0x_0+\gamma_1x_1+\gamma_3x_3)
    \end{pmatrix}.$$
    Since we are looking to $\cD_3(f)$, we have to consider the vanishing of all the $4\times4$ minors. With our usual enumeration from $0$ to $4$, we have essentially to consider the following minors (where $A^{i,j}$ denotes the minor obtained from the matrix $A$ by deleting the $i-$th row and the $j-$th column):
    $$(H_{\tilde{f}})^{2,2}\qquad (H_{\tilde{f}})^{2,3}\qquad (H_{\tilde{f}})^{3,3}.$$
    One can easily compute the determinants of these submatrices: by setting $p(x_0,x_1)$ as the determinant of the $2\times2$ matrix in the upper-left corner, in the order, we have 
    $$\det((H_{\tilde{f}})^{2,2})=-(2\beta_3x_3)^2\cdot p, \quad \det((H_{\tilde{f}})^{2,3})=-(\beta_5x_1)(2\beta_3x_3)\cdot p, \quad \det((H_{\tilde{f}})^{3,3})=-(\beta_5x_1)^2\cdot p.$$
    Observe now that $p\not\equiv0$, otherwise $\Pi$ would be contained in $\Sing(\cH_{\tilde{f}})$, which would have then dimension $2$, i.e. $\tilde{f}$ would be of TS type, against our assumptions. Moreover, in the case where $p\neq0$, then we get just the point $P_0$: all the other points in the intersection $\Pi\cap\Sing(\cH_f)$ are described by $p=0$, where $p$ is a polynomial of degree $2$ in two variables. We then get the claim.    
\end{proof}

\begin{corollary}
\label{COR:gnotastildef}
    Let $[g]$ be as in the Theorem \ref{THM:CO}. Then for any $2$-plane $\Lambda$ contained in $\cH_g$ there exists a point $P\in\cD_2(g)$ such that $\iota_g(P)=\Lambda$.
\end{corollary}
\begin{proof}
    Reasoning by contradiction, by the dichotomy given by Theorem \ref{THM:linearspaces}, we can assume that $g$ can be written as in Equation \eqref{EQ:special3fold}. Then, by the proof of the above Lemma \ref{LEM:bruttaftilde}, we have that there exists a $2$-plane $\Pi$, on which the components of dimension $1$ of the intersection with $\cD_3(g)$ are just two lines. Since $\cH_f=\cH_g$, with $f$ of expression as in \eqref{EQ:generalWR}, we have that these lines are of the form $V(L_i,L_j,L_k)$: the plane $\Pi$ has then to be defined by $V(L_i,L_j)$. This is not possible, since the plane $V(L_i,L_j)$ intersects $\Sing(\cH_g)$ in $4$ distinct lines.
\end{proof}

From this we finally get the following:
\begin{corollary}
\label{COR:constantlinesforiotag}
    In the hypotheses of Theorem \ref{THM:CO}, we have:
    \begin{itemize}
        \item $\cD_2(f)=\cD_2(g)$;
        \item any of the $20$ lines, which are components of $\Sing(\cH)$, has constant kernel with respect to $\iota_g$.   
    \end{itemize} 
\end{corollary}
\begin{proof}
    Since $\cH=\cH_g=\cH_f$, we know that there exist $15$ planes in $\cH_g$ of the form $\Pi_{ij}$: from Corollary \ref{COR:gnotastildef} we know that for each $\Pi_{ij}$ there is a point $P\in\cD_2(g)$ such that $\iota_g(P)=\Pi_{ij}$. Since $\cD_1(g)=\emptyset$ (see Remark \ref{REM:gnotTS}), we have that if $\Pi_{ij}\neq\Pi_{i'j'}$ then also the corresponding points $P$ and $P'$ are different. We have then $15$ distinct points in $\cD_2(g)$. Finally, since $\cD_2(g)\subseteq\Sing(\Sing(\cH))=\cD_2(f)$, which consists of $15$ points, we get the first claim.\\
    Take now one of the $20$ lines and denote it by $\ell$: we know by Lemma \ref{LEM:configuration} that it is contained in three $2$-planes, say $\Lambda_1,\Lambda_2,\Lambda_3$. We know by the above Corollary \ref{COR:gnotastildef} that there exists $Q_1,Q_2,Q_3\in\cD_2(g)$, such that $\Lambda_i=\iota_f(Q_i)$ (we need to distinguish these points or planes with a different notation, since a priori the two correspondences $\iota_f$ and $\iota_g$ are distinct). If we then consider a point $P\in\ell$, by symmetry, we have that $Q_1,Q_2,Q_3\in\iota_g(P)$. Since for $P$ general, $\iota_g(P)$ is a line (indeed we know by Remark \ref{REM:gnotTS} that any component of $\Sing(\cH)$ is not contained in $\cD_2(g)$), we have that $Q_1$, $Q_2$, and $Q_3$ are aligned on a line $r$ and we have $r:=\iota_g(\ell)$, i.e. $\ell$ has constant kernel with respect to $\iota_g$, as claimed.
\end{proof}

\subsection{Step 2}
\label{SUBS:coincideiota}
\ \\
Given $[f]$ and $[g]$ as in Theorem \ref{THM:CO}, with $\cH:=\cH_f=\cH_g$ we have shown so far that
$$\Sing(\cH)=\cD_3(f)=\cD_3(g)=\{20 \mbox{ lines}\}, \quad \cP:=\cD_2(f)=\cD_2(g)=\{15 \mbox{ points}\},$$
$$\cD_1(f)=\cD_1(g)=\emptyset;$$
moreover, each of the $20$ lines in $\Sing(\cH)$ has constant kernel, both with respect to $\iota_f$ and $\iota_g$. However, a priori the two correspondences induced by $\iota_f$ and $\iota_g$ do not coincide on the lines in $\Sing(\cH)$: the aim of this subsection is to show that they actually do.
Let us write the cubic $f$ of Waring Rank $6$, as in the expression \eqref{EQ:WRnormalform}, i.e.
$$f=\sum_{i=0}^4x_i^3+L^3, \quad \mbox{with } L=\sum_{i=0}^4a_ix_i.$$
One can then observe that from this expression of $f$, we get that 
$$\cP=\{P_0,\dots,P_4,Q_{01},\dots Q_{34}\},$$
where, by setting here $\{i,j,k,l,m\}=\{0,1,2,3,4\}$, we have that $P_i=V(x_j,x_k,x_l,x_m)$ (these are the $5$ coordinate points) and $Q_{ij}=V(L,x_k,x_l,x_m)$ (for example, we have $Q_{01}=[-a_1:a_0:0:0:0]$). We will refer to these points $Q_{ij}$ as {\it{mixed points}}. Let us also observe that if $\Pi$ is one of the $15$ $2$-planes defined as the kernels $\iota_f(P)$ with $P\in\cD_2(f)$, then its intersection with $\cP$ either consists of $6$ mixed points or it consists of $3$ mixed points and $3$ coordinate points.\\
For simplicity, we will refer to lines in $\Sing(\cH)$, points in $\cP$, and these $2$-planes as {\it{special}} lines, points and planes respectively.

The key step is the following:
\begin{lemma}
\label{LEM:keyforinvolution}
    Let $r$ be a special line and let $P$ be a special point in $r$. Then the kernel $\iota_g(P)$ does not intersect $r$.
\end{lemma}

\begin{proof}
    Let us denote $\Pi_P:=\iota_g(P)\simeq\bP^2$ and assume by contradiction that $\Pi_P\cap r=\{Q\}$ (of course $\Pi_P$ doesn't contain the entire line $r$, otherwise we would have $P\in\iota_g(P)$, contradicting the smoothness of $g$ by property \eqref{EQ:notdiagonal}). Since, by construction, $\Pi_P$ is a special plane, it easily follows that the intersection of $\Pi_P$ with a special line is a special point, i.e. $Q\in\cP$.\\
    Let us now observe that up to a linear projectivity, we can assume w.l.o.g. that $P$ and $Q$ are two coordinate points: let us set $P_0=P$ and $P_1=Q$ and denote the corresponding special planes by $\Pi_{P_0}:=\iota_g(P_0)\simeq\bP^2$ and $\Pi_{P_1}:=\iota_g(P_1)\simeq\bP^2$. Recall that we assumed that $P_1\in\Pi_{P_0}$, hence by symmetry we also that $P_0\in\Pi_{P_1}$. We know by Corollary \ref{COR:constantlinesforiotag} that $r$ has constant kernel with respect to $\iota_g$: let us denote $\tilde{r}=\iota_g(r)$, and observe that by construction $\tilde{r}=\Pi_{P_0}\cap\Pi_{P_1}$. Since $P_0,P_1\in r$ by construction, we clearly have that $P_0,P_1\not\in\tilde{r}$, otherwise, as usual, from property \eqref{EQ:notdiagonal}, we would have a singular point for $V(g)$. \\
    One sees that in $\Pi_{P_1}$ there are two special lines passing through $P_0$ (observe that these are different from $r$, otherwise we would have $P_1\in\iota_g(P_1)$ which is again not possible).
    Let us then take $\ell\subset\Pi_{P_1}$ as one of these two special lines and let us denote by $A$ the intersection $\ell\cap\tilde{r}$. Since $\ell$ is a special line, by Corollary \ref{COR:constantlinesforiotag}, it has constant kernel (with respect to $\iota_g$), denoted by $\tilde{\ell}$. By symmetry, one can observe that $\tilde{\ell}\subset\Pi_{P_0}$ (since $P_0\in\ell$) and $P_1\in\tilde{\ell}$ (since $\ell\subset\Pi_{P_1}$). In other words, we can write $\tilde{\ell}=\bP(\left\langle P_1,B\right\rangle)$, where $B:=\tilde{\ell}\cap\tilde{r}$, clearly distinct from $A$, otherwise we would have $A\in\iota_g(A)$, which is again not possible. Observe that, being intersection of special lines, the points $A$ and $B$ are special points. \\
    In a similar way, let us consider in $\Pi_{P_0}$ the line $\ell':=\bP(\left\langle P_1,A\right\rangle)$: one can easily see that $P_0,B\in\iota_g(P_1)\cap\iota_g(A)$, i.e. $\ell'$ is a special and constant line with kernel $\iota_g(\ell'):=\tilde{\ell'}=\bP(\left\langle P_0,B\right\rangle)$.\\

    \noindent
    {\underline{Claim: $A$ and $B$ are coordinate points}}.
    Before this lemma, we noticed that a special plane either contains $3$ coordinate points or none of them. Since both $\Pi_{P_0}$ and $\Pi_{P_1}$ contain one coordinate point (resp. $P_1$ and $P_0$), we have that there are three coordinate points both in $\Pi_{P_0}$ and in $\Pi_{P_1}$.\\
    Let us focus on $\Pi_{P_0}$: since $P_1\in\Pi_{P_0}$ we have that $\Pi_{P_0}=V(x_i,x_j)$ for $i,j\in\{0,2,3,4\}$ distinct indices. In particular, in $\Pi_{P_0}$, we have $3$ special lines intersecting the locus $\cP$ in one mixed point and two coordinate points and one special line $V=(x_i,x_j,L)$ intersecting $\cP$ in three distinct mixed points.\\   Furthermore, let us observe that a special line with three distinct mixed points is of the form $V(L,x_i,x_j)$ for distinct indices $i,j\in\{0,\dots,4\}$: it is contained in $3$ distinct special planes, but only one of such planes, namely $V(x_i,x_j)$, intersects $\cP$ in three coordinate points and $3$ mixed point. Since, in our case, the line $\tilde{r}$ is contained in two special planes of this type, it has to intersect $\cP$ in one mixed point and two coordinate points. W.l.o.g. (up to projectivity) we can then assume that such two coordinate points are $P_2,P_3$: we have then $\tilde{r}\cap\cP=\{P_2,P_3,Q_{23}\}$, i.e. $\tilde{r}=V(x_0,x_1,x_4)$. To conclude, let us observe that the line $\bP(\left\langle P_1,Q_{23}\right\rangle)$ is not a special line: the special lines in $\Pi_{P_0}$ passing through $P_1$ have to intersect the line $\tilde{r}$ in the coordinate points. In other words, we can assume that $A=P_2$ and $B=P_3$.
   \smallskip

   Now, we are going to use all these information to reconstruct the cubic $g$ and to find a contradiction. Indeed, what we have shown so far allows us to get some vanishings of specific differential operator (i.e. vanishings conditions in the apolar ring $\cA_g$). Recalling the identification $\bP^4\simeq\bP(A^1_g)$ (see Subsection \ref{SUBS:jacobianandapolar}) and seeing then the coordinate points $P_i$ as the elements $[y_i]=[\frac{\partial}{\partial x_i}]\in\bP(A_g^1)$, for example, we have that   
   \begin{equation}
   \label{EQ:vanish1}
   P_0\in\Pi_{P_1}=\iota_g(P_1) \quad \Longrightarrow \quad y_0y_1=0 \ \mbox{in } A^1_g.
   \end{equation}
   In other words, in the expression of $g$ the product $x_0x_1$ can't appear.
   \smallskip
   \ \\
   By construction, in the same way, we get by symmetry that $P_0,P_1,P_3\in\iota_g(P_2)$ and $P_0,P_1,P_2\in\iota_g(P_3)$. Hence, we have
   \begin{equation}
   \label{EQ:vanish2}    
       \frac{\partial^2}{\partial x_0\partial x_2}(g) \ = \ \frac{\partial^2}{\partial x_1\partial x_2}(g) \ = \ \frac{\partial^2}{\partial x_0\partial x_3}(g) \ = \ \frac{\partial^2}{\partial x_1\partial x_3}(g) \ = \ \frac{\partial^2}{\partial x_2\partial x_3}(g) \ = \ 0.
   \end{equation}
   Looking at the other coordinate point, namely $P_4$, let us prove the following:
   \smallskip
   
   {\underline{Claim: $\iota_g(P_4)=V(x_4,L)$}}. Let us start by observing that on the line $r$ we have also the mixed point $Q_{01}$ as point in $\cD_2(g)$. By construction, the special plane $\iota_g(Q_{01})$ contains the line $\tilde{r}$, and so the points $P_2,P_3,Q_{23}$. In particular, since it contains two coordinate points (namely $P_2$ and $P_3$), it contains, as seen above, a third coordinate point. It is easy to see that such a point has to be $P_4$: indeed, if we assume that $P_0\in\iota_g(Q_{01})$ then by symmetry we would get $Q_{01},P_1\in\Pi_{P_0}$, i.e. $P_0\in r\subset\Pi_{P_0}=\iota_g(P_0)$, contradicting the smoothness of $g$ (see again property \eqref{EQ:notdiagonal}). In the same way, we get that $P_1\not\in\iota_g(Q_{01})$. The same happens for the point $Q_{23}\in\tilde{r}$: we have  $$P_4\in\iota_g(Q_{01})\cap\iota_g(Q_{23}).$$
   By symmetry, we get that $\iota_g(P_4)$ contains the points $Q_{01}$ and $Q_{23}$: the only special plane containing these two points is $V(x_4,L)$, as claimed.
   \smallskip

   In particular, we have that \begin{equation}
       \label{EQ:KerP4}
        Q_{ij}\in\iota_g(P_4) \ \mbox{for} \ i\in\{0,1,2\}, \  j\in\{1,2,3\} \ \mbox{and} \ i<j.
    \end{equation}
   Since the point $Q_{ij}$ can be thought as $[-a_jy_i+a_iy_j]\in\bP(A_g^1)$, we have that the above conditions \eqref{EQ:KerP4}, translated into the apolar ring $\cA_g$, give the following relations:
   \begin{equation}
       \label{EQ:vanish3}
   a_j\frac{\partial^2}{\partial x_i\partial x_4}(g)=a_i\frac{\partial^2}{\partial x_j\partial x_4}(g), \qquad \mbox{for }i\in\{0,1,2\}, \ j\in\{1,2,3\}, \ i<j.
   \end{equation}
   \smallskip

   Let us now reconstruct the cubic $g$ (or, better, a family) which doesn't satisfy the properties it has to, getting then a contradiction.\\
   First of all, from the vanishings \eqref{EQ:vanish1} and \eqref{EQ:vanish2}, we can write:
   \begin{equation}
   \label{EQ:firststep}
        g=\sum_{k=0}^4\alpha_kx_k^3+x_4^2\left(\sum_{k=0}^3\beta_kx_k\right)+x_4\left(\sum_{k=0}^3\gamma_kx_k^2\right).
   \end{equation}
   We get
   $$g_4:=\frac{\partial}{\partial x_4}(g)=3\alpha_4x_4^2+2x_4\left(\sum_{k=0}^3\beta_kx_k\right)+\left(\sum_{k=0}^3\gamma_kx_k^2\right).$$
    Then by the relations \eqref{EQ:vanish3}, we get for $0\leq i<j\leq3$
    $$a_i\frac{\partial}{\partial x_j}(g_4) \ = \ a_j\frac{\partial}{\partial x_i}(g_4)\longrightarrow a_i(2\beta_j x_4+2\gamma_j x_j) \ = \ a_j(2\beta_i x_4+2\gamma_i x_i).$$
    Hence, since by assumptions the coefficients $a_i$ are all different from zero (otherwise $f$ would be of TS type) we get that 
    $$\gamma_k=0 \mbox{ for all } k\in\{0,\dots,3\} \mbox{ and } a_i\beta_j=a_j\beta_i \mbox{ for } 0\leq i < j\leq 3.$$
    Then, one sees that if there exists an index $m$ such that $\beta_m=0$, then $\beta_k=0$ for all $k\in\{0,\dots,3\}$. In such a case, since also $\gamma_k=0$ for each $k$, we would get $g=\sum_{k=0}^4\alpha_kx_k^3$, i.e. $g$ would be of Thom-Sebastiani type, against our assumptions. Hence, each of the coefficients $\beta_i$ is not zero. Thus, one can easily see that 
    $$\exists \lambda\in\bC\setminus\{0\} \mbox{ such that } \beta_k=\lambda a_k \mbox{ for each } k\in\{0,\dots,3\}.$$ 
    By replacing $\alpha_4$ with $\alpha_4+\lambda a_4$, we can then write
    \begin{equation}
        \label{EQ:finalstep}
        g=\sum_{k=0}^4\alpha_k x_k^3+\lambda Lx_4^2 \quad \mbox{with } L=\sum_{k=0}^4a_kx_k.
    \end{equation}
    Observe now that $\alpha_k\neq0$ for all $k\in\{0,\dots,3\}$: indeed, if $\alpha_i=0$ for some index $i\in\{0,\dots,3\}$ then the coordinate point $P_i$ would be singular for $V(g)$, which is not possible, since we are assuming $V(g)$ to be smooth.
    \smallskip
    
    For $g$ defined as in Equation \eqref{EQ:finalstep}, we can then write the associated Hessian matrix
    $$H_g=2\begin{pmatrix}3\alpha_0x_0&0&0&0&\lambda a_0x_4\\0&3\alpha_1x_1&0&0&\lambda a_1x_4\\0&0&3\alpha_2x_2&0&\lambda a_2x_4\\
    0&0&0&3\alpha_3x_3&\lambda a_3x_4\\\lambda a_0x_4&\lambda a_1x_4&\lambda a_2x_4&\lambda a_3x_4&(\lambda L+3\alpha_4x_4+2\lambda a_4x_4)\end{pmatrix}.$$
    To conclude, since we proved that $\cD_2(g)=\cD_2(f)=\{P_0,\dots,P_4,Q_{01},\dots,Q_{34}\}$, we should have that $\Rank(H_g(Q_{04}))\leq 2$. Since $Q_{04}=[-a_4:0:0:0:a_0]$, we have
    $$H_g(Q_{04})=2\begin{pmatrix}-3\alpha_0a_4&0&0&0&\lambda a_0^2\\0&0&0&0&\lambda a_1a_0\\0&0&0&0&\lambda a_2a_0\\
    0&0&0&0&\lambda a_3a_0\\\lambda a_0^2&\lambda a_1a_0&\lambda a_2a_0&\lambda a_3a_0&a_0(3\alpha_4+2\lambda a_4)\end{pmatrix}.$$
    Consider, for example, the $(3\times3)$-minor
    $$\begin{vmatrix}
        -3\alpha_0 a_4&0&\lambda a_0^2\\0&0&\lambda a_1a_0\\\lambda a_0^2&\lambda a_1a_0&a_0(3\alpha_4+2\lambda a_4)
    \end{vmatrix}=
    3\lambda^2a_0^2a_1^2\alpha_0a_4.$$
    Since, by assumption $a_i\neq0$ for every $i=0,\dots,4$ and we have shown that also $\lambda\neq0$ and $\alpha_j\neq0$ for $j=0,\dots,3$, we have that this minor does not vanish, getting a contradiction, and concluding the proof. 
\end{proof}

From the above lemma \ref{LEM:keyforinvolution}, we finally get the following:

\begin{proposition}
\label{PROP:sameinvolution}
    Given $f$ and $g$ as in Theorem \ref{THM:CO}, we have that 
    $$\iota_f=\iota_g \qquad \mbox{on }\Sing(\cH).$$
\end{proposition}

\begin{proof}
    We know that the $20$ special lines, components of $\Sing(\cH)$ have constant kernels with respect to $\iota_f$ and by Corollary \ref{COR:constantlinesforiotag} also with respect to $\iota_g$. This means that given a special line $\ell$, there exist two special lines $\ell_f$ and $\ell_g$, such that $\iota_f(\ell)=\ell_f$ and $\iota_g(\ell)=\ell_g$. To prove the Proposition, it is then enough to show that $\ell_f=\ell_g$. Indeed, let $P$ be a point in $\cD$, since it belongs to three different special lines, say $\ell_1,\ell_2,\ell_3$ then we would have that $\iota_f(P)$ is the special plane containing $\iota_f(\ell_i)$ for $i=1,2,3$ and $\iota_g(P)$ is the special plane containing $\iota_g(\ell_i)$ for $i=1,2,3$. If we prove that $\iota_f(\ell_i)=\iota_g(\ell_i)$ for each index $i$, then it follows that $\iota_f(P)=\iota_g(P)$.\\
    Let $\ell$ be a special line and $P_1,P_2,P_3$ the $3$ special points on $\ell$, since $\ell$ has constant kernel with respect to both the correspondence, we can write
    $$\ell_g\subset\iota_g(P_i) \ \forall i=1,2,3 \quad \mbox{i.e.} \ \ell_g=\iota_g(P_1)\cap\iota_g(P_2)\cap\iota_g(P_3)$$
    and the same with $\ell_f$ and $\iota_f$.
    From the description given for $\iota_f$ (see Lemma \ref{LEM:iotaf}), we have that, in this situation, $\iota_f(P_i)\cap \ell=\emptyset$ for all $i=1,2,3$ and, by Lemma \ref{LEM:keyforinvolution}, the same happens with respect to $\iota_g$, i.e.
    $$\iota_g(P_i)\cap \ell=\emptyset.$$
    It is then enough to observe that among the $15$ special planes only $3$ of them, namely $\Pi_1,\Pi_2,\Pi_3$ do not intersect the fixed line $\ell$: both the kernel $\ell_f$ and $\ell_g$ have to coincide with the intersection $\Pi_1\cap\Pi_2\cap\Pi_3$, proving the proposition.
\end{proof}

\subsection{Step 3}
\label{SUBS:againWR}
\ \\
Let us now complete the proof of Theorem \ref{THM:CO}. In the previous subsection, with Proposition \ref{PROP:sameinvolution}, we proved that the two correspondences $\iota_f$ and $\iota_g$ act in the same exact way on $\Sing(\cH)$: let us denote by $\iota$ this correspondence. In particular, for $f$ written as in the expression \eqref{EQ:generalWR}, we have that
$$\iota(V(L_i,L_j,L_k))=V(L_l,L_m,L_n) \quad \mbox{and} \quad \iota(V(L_i,L_j,L_k,L_l))=V(L_m,L_n),$$
where $\{i,j,k,l,m,n\}=\{0,1,2,3,4,5\}$.\\
Here, we exploit such a property to reconstruct the cubic $[g]\in\bP(S^3)$ and to prove the following:

\begin{lemma}
    \label{LEM:gisWR}
    Let $[f]$ and $[g]$ be as in Theorem \ref{THM:CO}. Then 
    $$[g] \ \mbox{itself belongs to } \cW_6.$$
\end{lemma}

\begin{proof}
    For the proof, let us consider $f$ as in the expression \eqref{EQ:WRnormalform}.\\
    First of all, observe that since the entries of the Hessian matrix of $g$ are linear forms with variables $x_0,\dots,x_4$, it can be written as 
    $$H_g=\sum_{i=0}^4x_iA_i,$$
    where $A_i\in\mathrm{Mat}(5,\bK)$ are square symmetric matrices with coefficient in the field $\bK$. By construction, we then have that $H_g(P_i)=A_i$ and so $H_g(P_i)\cdot P_j=(A_i)^j$, i.e. the $j$-th column of the matrix $A_i$, where $P_i$ is, as usual, the $i$-th coordinate point.\\
    Considering the kernel map $\iota$, we know, for example, that 
    \begin{equation}
        \label{EQ:cond1}
        \iota(P_i)=V(L,x_i)\ni Q_{jk} \ \mbox{for } j,k\in\{0,\dots,4\}\setminus \{i\}.
    \end{equation}
    From this, one easily gets (with the usual numeration from $0$ to $4$)
    \begin{equation}
        \label{EQ:relationcoeff1}
        a_k(A_i)^j=a_j(A_i)^k \qquad \mbox{for all } i,j,k \mbox{ pairwise distinct}.
    \end{equation}
    
    Furthermore, denoting by $(A)_k$ the $k$-th row of the matrix $A$, let us now observe that the element $(A_i)^j_k$ coincides with the coefficient of $x_i$ in the second partial derivatives $\frac{\partial^2f}{\partial x_j\partial x_k}$, i.e.
    $$(A_i)^j_k=\frac{\partial^3f}{\partial x_i\partial x_j\partial x_k}.$$
    From this and the Schwartz' relation, we get that 
    \begin{equation}
        \label{EQ:schwartz}
        (A_i)^j_k\in \bK \quad \mbox{is fixed by any permutation of the pairwise distinct indices }i,j,k.
    \end{equation}
    Observe that there exists an index $\alpha$ such that the matrix $A_{\alpha}$ has an entry different from zero: there exist distinct indices $\beta,\gamma\in\{0,\dots,4\}\setminus\{\alpha\}$ such that $(A_{\alpha})^{\beta}_{\gamma}\neq 0$. We can then take a value $\lambda\in\bK^*$ which allows us to write
    \begin{equation}
        \label{EQ:startingcoeff}
        (A_{\alpha})^{\beta}_{\gamma}=\lambda a_{\alpha}a_{\beta}a_{\gamma}, 
    \end{equation}
    where the $a_i$'s are the coefficients in the expression \eqref{EQ:WRnormalform} of $f$ (recall that by assumption $a_i\neq0$ for all i.\\
    From this, using relation \eqref{EQ:relationcoeff1} and the fact \eqref{EQ:schwartz}, one can obtain all the coefficients $(A_i)^j_k$ for any triple of indices $i,j,k\in\{0,\dots,5\}$ pairwise distinct. \\
    Again, by \eqref{EQ:relationcoeff1}, for distinct indices $i,j$, we have
    $$(A_i)^j_j=\frac{a_j}{a_k}(A_i)^k_j=\lambda a_ia_j^2.$$
    Hence, one gets all the coefficient $(A_i)^j_k$ with at most two indices among $\{i,j,k\}$ equal. \\
    Summarizing, the coefficients of the matrices $A_i$ can be written as
    \begin{equation}
        \label{EQ:laststep}
        (A_i)^j_k=\lambda a_ia_ja_k \mbox{ for a fixed }\lambda\in\bK^* \mbox{and }i,j,k\in\{0,\dots,4\} \mbox{ not all the same}.
    \end{equation}
    For example, we can write
    $$A_0=\lambda\begin{pmatrix}
        *&a_0^2a_1& a_0^2a_2&a_0^2a_3&a_0^2a_4\\
        a_0^2a_1&a_0a_1^2& a_0a_1a_2&a_0a_1a_3&a_0a_1a_4\\
        a_0^2a_2&a_0a_1a_2& a_0a_2^2&a_0a_2a_3&a_0a_2a_4\\
        a_0^2a_3&a_0a_1a_3& a_0a_2a_3&a_0a_3^2&a_0a_3a_4\\
        a_0^2a_4&a_0a_1a_4& a_0a_2a_4&a_0a_3a_4&a_0a_4^2\\
    \end{pmatrix},$$
    and in an analogous way for the other matrices $A_i$. Hence, there exist coefficient $b_i\in\bK$ for $i=0,\dots,4$, such that 
    \begin{equation}
        \label{EQ:Hessg}
        H_g=\lambda\left(\begin{pmatrix}            6b_0x_0&0&0&0&0\\0&6b_1x_1&0&0&0\\0&0&6b_2x_2&0&0\\0&0&0&6b_3x_3&0\\0&0&0&0&6b_4x_4
        \end{pmatrix}
        +\underbrace{\sum_{i=0}^4a_ix_i}_{L}(\underline{a}^T\underline{a})\right),
    \end{equation}
    where $\underline{a}=(a_0,\dots,a_4)$ and the $a_i$'s are, as above, the coefficients of $L$ in the expression of $f$.
    Since the coefficient $\lambda$ does not play any particular role, in what follows, we can assume w.l.o.g. $\lambda=1$. Let us now reconstruct the cubic polynomial $g$, in the most obvious way. First of all let us consider the second partial derivative $\frac{\partial ^2 g}{\partial x_0^2}$:
    $$\frac{\partial ^2 g}{\partial x_0^2}=(H_g)^0_0=6b_0x_0+a_0^2\sum_{i=0}^4a_ix_i.$$
    Then integrating with respect to $x_0$, we get
    \begin{equation}
        \label{EQ:firstpartialx0incognite}
        \frac{\partial g}{\partial x_0}=3b_0x_0^2+\frac{a_0^3}{2}x_0^2+a_0^2a_1x_0x_1+a_0^2a_2x_0x_2+a_0^2a_3x_0x_3+a_0^2a_4x_0x_4+\sum_{i,j\in\{1,\dots,4\}, \ i\leq j}\alpha_{ij}x_ix_j.
    \end{equation}
    Let us now compute the coefficients $\alpha_{ij}$ in the relation \eqref{EQ:firstpartialx0incognite}. For example, we get
    $$\mbox{by \eqref{EQ:firstpartialx0incognite}}: \frac{\partial^2g}{\partial x_0\partial x_1}=a_0^2a_1x_0+2\alpha_{11}x_1+\alpha_{12}x_2+\alpha_{13}x_3+\alpha_{14}x_4$$
    $$\mbox{by \eqref{EQ:Hessg}}: \frac{\partial^2g}{\partial x_0\partial x_1}=(H_g)^0_1=a_0a_1\cdot L=a_0^2a_1x_0+a_0a_1^2x_1+a_0a_1a_2x_2+a_0a_1a_3x_3+a_0a_1a_4x_4.$$
    Thus, we find
    $$\alpha_{11}=\frac{a_0a_1^2}{2}, \quad \alpha_{12}=a_0a_1a_2, \quad \alpha_{13}=a_0a_1a_3, \quad \alpha_{14}=a_0a_1a_4.$$
    In the same way, one obtains all the parameters $\alpha_{ij}$ appearing in Equation \eqref{EQ:firstpartialx0incognite} by using the other second partial derivatives. Then, one can integrate again with respect to $x_0$ and with the same procedure obtain the cubic $g$, as follows:

    $$g=\sum_{i=0}^4b_ix_i^3+\sum_{i=0}^4\frac{a_i^3}{6}x_i^3+\sum_{i,j\in\{0,\dots,4\}, \ i\neq j}\frac{a_i^2a_j}{2}x_i^2x_j+\sum_{i,j,k\in\{0,\dots,4\}, \ i\leq j \leq k}a_ia_ja_kx_ix_jx_k.$$
    One can then see that $g$ is itself of Waring Rank $6$, i.e. an element of $\cW_6$, since we can write:
    $$6g=6\sum_{i=0}^4b_ix_i^3+(a_0x_0+a_1x_1+a_2x_2+a_3x_3+a_4x_4)^3.$$    
\end{proof}

As said above, by using now Theorem \ref{THM:WRn+2}, the proof of Theorem \ref{THM:CO} is completed.

\section{Hessians and cones}
\label{SEC:hessiancones}
In this last section, we prove Theorem D from the Introduction:

\begin{theorem}
\label{THM:hessianinjacobians}
    Let $[f]\in\bP(S^d_n)$ be a homogeneous polynomial of degree $d\geq3$ in $n+1$ variables, with $n\geq1$. It holds
    $$x_i\partial_{x_j}h_f - (d-2)\delta_{ij}h_f \in (J_f)^{(n+1)(d-2)},$$
    where $\delta_{ij}$ is the Kronecker symbol.
\end{theorem}

\begin{proof}
    Let us fix, for example, $i=0$. By denoting by $f_i$ the first partial derivative of $f$ with respect to $x_i$ and by $f_{ij}$ the second partial derivative of $f$ with respect to $x_i$ and $x_j$, we have
    $$x_0h_f=\det\begin{pmatrix}
        x_0f_{00}&x_0f_{01}&\cdots&x_0f_{0n}\\f_{01}&f_{12}&\cdots&f_{1n}\\\vdots&\vdots&\cdots&\vdots\\f_{0n}&f_{1n}&\cdots&f_{nn}
    \end{pmatrix}$$
    Denoting by $r_i$ the $i$-th row of such a matrix for $i=0,\cdots,n$, let us substitute $r_0$ with $r_0+\sum_{i\geq1}x_ir_i$. Since by the Euler differential Identity we have $(d-1)f_i=\sum_{i=0}^nx_if_{ij}$, we get
    $$x_0h_f=\det\begin{pmatrix}
        (d-1)f_0&(d-1)f_1&\cdots&(d-1)f_n\\f_{01}&f_{12}&\cdots&f_{1n}\\\vdots&\vdots&\cdots&\vdots\\f_{0n}&f_{1n}&\cdots&f_{nn}
    \end{pmatrix}.$$
    Hence, we have 
    $$x_0h_f=(d-1)\det(B), \quad \mbox{where} \quad B=\begin{pmatrix}
        f_0&f_1&\cdots&f_n\\f_{01}&f_{12}&\cdots&f_{1n}\\\vdots&\vdots&\cdots&\vdots\\f_{0n}&f_{1n}&\cdots&f_{nn}
    \end{pmatrix}.$$
    Let us consider the case where $j\neq i$: from the above, we have 
    $$x_0\partial_{x_j}h_f=(d-1)\partial_{x_j}\det(B)$$
    and $\partial_{x_j}\det(B)=\sum_{k=0}^n\det(B_k)$, where $B_k$ denotes the matrix $B$ with the $k$-th row derived by $x_j$. One then notices that if $k=0$, two rows of $B_0$ coincide, namely the $0$-th row and the $j$-th row, so its determinant vanishes. For $k>0$, considering the Laplace expansion with respect to the $0$-th row, one easily sees that the determinant $\det(B_k)$ lives in the Jacobian ideal of $f$. Hence, we have proved that, for $i\neq j$
    $$x_i\partial_{x_j}h_f\in (J_f)^{(n+1)(d-2)}.$$
    Let us finally consider the case where $j=0$: let us prove that $x_0\partial_{x_0}h_f=(d-2)h_f+\alpha$, where $\alpha\in (J_f)^{(n+1)(d-2)}$.
    From above, we have $x_0h_f=(d-1)\det(B)$, and so, deriving with respect to $x_0$, we get:
    $$h_f+x_0\partial_{x_0}h_f=(d-1)\partial_{x_0}(\det B) \quad \rightarrow \ x_0\partial_{x_0}h_f=(d-1)\partial_{x_0}(\det B)-h_f.$$
    Moreover, reasoning as before one sees that $\partial_{x_0}(\det(B))=h+\alpha$, with $\alpha\in(J_f)^{(n+1)(d-2)}$: we get the claim.
\end{proof}

From this, one can easily get the following:
\begin{corollary}
\label{COR:conclusioncones}
    Let $[f]\in\bP(S^d_n)$ be a homogeneous polynomial of degree $d\geq3$ in $n+1$ variables, with $n\geq1$. If the associated hypersurface $V(f)\subset\bP^n$ is smooth, then $\cH_f$ is not a cone.
\end{corollary}

\begin{proof}
    Let us assume by contradiction that $\cH_f$ is a cone. This means that, up to a change of coordinates, we have that $\partial_{x_0} h_f\equiv0$. By Theorem \ref{THM:hessianinjacobians}, we get that $h_f$ lies in the Jacobian ideal of $f$. By Proposition \ref{PROP:soclejacobian}, this is not possible in the case where the associated hypersurface $V(f)\subset\bP^n$ is smooth.
\end{proof}

\bibliographystyle{amsalpha}

\begin{thebibliography}{AAAAA11}

\bib{AGV}{book}{
   author={Arnold, V. I.},
   author={Gusein-Zade, S. M.},
   author={Varchenko, A. N.},
   title={Singularities of differentiable maps. Vol. I},
   series={Monographs in Mathematics},
   volume={82},
   publisher={Birkh\"{a}user Boston, Inc., Boston, MA},
   date={1985},
   pages={xi+382},
   isbn={0-8176-3187-9},
   doi={10.1007/978-1-4612-5154-5},
}


\bib{AR}{book}{
   author={Adler, A.},
   author={Ramanan, S.},
   title={Moduli of abelian varieties},
   series={Lecture Notes in Mathematics},
   volume={1644},
   publisher={Springer-Verlag, Berlin},
   date={1996},
   pages={vi+196},
   isbn={3-540-62023-0},
   doi={10.1007/BFb0093659},
}


\bib{Beo}{article}{
   author={Beorchia, V.},
   title={Generic injectivity of Hessian maps of ternary forms},
   journal={arXiv:2406.05423},
   volume={},
   date={2024},
   number={},
   pages={},
   issn={},
}




\bib{Wa20}{article}{
   author={de Bondt, M.},
   author={Watanabe, J.},
   title={On the theory of Gordan-Noether on homogeneous forms with zero
   Hessian (improved version)},
   conference={
      title={Polynomial rings and affine algebraic geometry},
   },
   book={
      series={Springer Proc. Math. Stat.},
      volume={319},
      publisher={Springer, Cham},
   },
   date={2020},
   pages={73--107},
   doi={10.1007/978-3-030-42136-6\_3},
}



\bib{BM}{article}{
   author={Boralevi, A.},
   author={Mezzetti, E.},
   title={Pencils of singular quadrics of constant rank and their orbits},
   journal={Rend. Istit. Mat. Univ. Trieste},
   volume={54},
   date={2022},
   pages={Art. No. 16, 22},
   issn={0049-4704},
   review={\MR{4595173}},
}


\bib{BF}{article}{
   author={Bricalli, D.},
   author={Favale, F. F.},
   title={Lefschetz properties for jacobian rings of cubic fourfolds and
   other Artinian algebras},
   journal={Collect. Math.},
   volume={75},
   date={2024},
   number={1},
   pages={149--169},
   issn={0010-0757},
   doi={10.1007/s13348-022-00382-5},
}



\bib{BFGre}{article}{
   author={Bricalli, D.},
   author={Favale, F. F.},
   title={Standard Artinian algebras and Lefschetz properties: a geometric
   approach},
   conference={
      title={Deformation of Artinian algebras and Jordan type},
   },
   book={
      series={Contemp. Math.},
      volume={805},
      publisher={Amer. Math. Soc., Providence, RI},
   },
   isbn={978-1-4704-7356-3},
   isbn={9781470476656},
   date={2024},
   pages={125--138},
   review={\MR{4802593}},
   doi={10.1090/conm/805/16130},
}


\bib{BFP}{article}{
   author={Bricalli, D.},
   author={Favale, F. F.},
   author={Pirola, G. P.},
   title={A theorem of Gordan and Noether via Gorenstein rings},
   journal={Selecta Math. (N.S.)},
   volume={29},
   date={2023},
   number={5},
   pages={Paper No. 74, 25},
   issn={1022-1824},
   doi={10.1007/s00029-023-00882-7},
}


\bib{BFP2}{article}{
   author={Bricalli, D.},
   author={Favale, F. F.},
   author={Pirola, G. P.},
   title={On the Hessian of cubic hypersurfaces},
   journal={Int. Math. Res. Not. IMRN},
   date={2024},
   number={10},
   pages={8672--8694},
   issn={1073-7928},
   review={\MR{4749181}},
   doi={10.1093/imrn/rnad324},
}



\bib{BFP3}{article}{
   author={Bricalli, D.},
   author={Favale, F. F.},
   author={Pirola, G. P.},
   title={On the irreducibility of Hessian loci of cubic hypersurfaces},
   journal={Forum Math. Sigma},
   volume={13},
   date={2025},
   pages={Paper No. e80, 36},
   doi={10.1017/fms.2025.36},
}

\bib{BBKT}{article}{
   author={Buczy\'{n}ska, W.},
   author={Buczy\'{n}ski, J.},
   author={Kleppe, J.},
   author={Teitler, Z.},
   title={Apolarity and direct sum decomposability of polynomials},
   journal={Michigan Math. J.},
   volume={64},
   date={2015},
   number={4},
   pages={675--719},
   issn={0026-2285},
   review={\MR{3426613}},
   doi={10.1307/mmj/1447878029},
}


\bib{CO}{article}{
   author={Ciliberto, C.},
   author={Ottaviani, G.},
   title={The Hessian map},
   journal={Int. Math. Res. Not. IMRN},
   date={2022},
   number={8},
   pages={5781--5817},
   issn={1073-7928},
   doi={10.1093/imrn/rnaa288},
}


\bib{COCD}{article}{
   author={Ciliberto, C.},
   author={Ottaviani, G.},
   author={Caro, J.},
   author={Duque-Rosero, J.}
   title={The general ternary form can be recovered by its Hessian},
   journal={arXiv:2406.05382},
   volume={},
   date={2024},
   number={},
   pages={},
   issn={},
}


\bib{DVG}{article}{
   author={Dardanelli, E.},
   author={van Geemen, B.},
   title={Hessians and the moduli space of cubic surfaces},
   conference={
      title={Algebraic geometry},
   },
   book={
      series={Contemp. Math.},
      volume={422},
      publisher={Amer. Math. Soc., Providence, RI},
   },
   date={2007},
   pages={17--36},
   doi={10.1090/conm/422/08054},
}

\bib{Dim}{article}{
   author={Dimca, A.},
   title={A geometric approach to the classification of pencils of quadrics},
   journal={Geom. Dedicata},
   volume={14},
   date={1983},
   number={2},
   pages={105--111},
   issn={0046-5755},
   review={\MR{708627}},
   doi={10.1007/BF00181618},
}



\bib{DGI}{article}{
   author={Dimca, A.},
   author={Gondim, R.},
   author={Ilardi, G.},
   title={Higher order Jacobians, Hessians and Milnor algebras},
   journal={Collect. Math.},
   volume={71},
   date={2020},
   number={3},
   pages={407--425},
   issn={0010-0757},
   doi={10.1007/s13348-019-00266-1},
}




\bib{Dol}{book}{
   author={Dolgachev, I. V.},
   title={Classical algebraic geometry},
   note={A modern view},
   publisher={Cambridge University Press, Cambridge},
   date={2012},
   pages={xii+639},
   isbn={978-1-107-01765-8},
   doi={10.1017/CBO9781139084437},
}

\bib{GR09}{article}{
   author={Garbagnati, A.},
   author={Repetto, F.},
   title={A geometrical approach to Gordan-Noether's and Franchetta's
   contributions to a question posed by Hesse},
   journal={Collect. Math.},
   volume={60},
   date={2009},
   number={1},
   pages={27--41},
   issn={0010-0757},
   doi={10.1007/BF03191214},
}

\bib{Fed}{article}{
   author={Fedorchuk, M.},
   title={Direct sum decomposability of polynomials and factorization of
   associated forms},
   journal={Proc. Lond. Math. Soc. (3)},
   volume={120},
   date={2020},
   number={3},
   pages={305--327},
   issn={0024-6115},
   review={\MR{4008372}},
   doi={10.1112/plms.12293},
}


\bib{GR}{article}{
   author={Gondim, R.},
   author={Russo, F.},
   title={On cubic hypersurfaces with vanishing hessian},
   journal={J. Pure Appl. Algebra},
   volume={219},
   date={2015},
   number={4},
   pages={779--806},
   issn={0022-4049},
   doi={10.1016/j.jpaa.2014.04.030},
}



\bib{GN}{article}{
   author={Gordan, P.},
   author={N\"{o}ther, M.},
   title={Ueber die algebraischen Formen, deren Hesse'sche Determinante
   identisch verschwindet},
   language={German},
   journal={Math. Ann.},
   volume={10},
   date={1876},
   number={4},
   pages={547--568},
   issn={0025-5831},
   doi={10.1007/BF01442264},
}




\bib{bookLef}{book}{
   author={Harima, T.},
   author={Maeno, T.},
   author={Morita, H.},
   author={Numata, Y.},
   author={Wachi, A.},
   author={Watanabe, J.},
   title={The Lefschetz properties},
   series={Lecture Notes in Mathematics},
   volume={2080},
   publisher={Springer, Heidelberg},
   date={2013},
   pages={xx+250},
   isbn={978-3-642-38205-5},
   isbn={978-3-642-38206-2},
   doi={10.1007/978-3-642-38206-2},
}

\bib{Hes59}{article}{
   author={Hesse, O.},
   title={Zur Theorie der ganzen homogenen Functionen},
   volume={56},
   publisher={J. Reine Angew. Math.},
   date={1859},
   pages={263--269}
}



\bib{Hut}{article}{
   author={Hutchinson, J. I.},
   title={The Hessian of the cubic surface. II},
   journal={Bull. Amer. Math. Soc.},
   volume={6},
   date={1900},
   number={8},
   pages={328--337},
   issn={0002-9904},
   doi={10.1090/S0002-9904-1900-00716-6},
}


\bib{Rus}{book}{
   author={Russo, F.},
   title={On the geometry of some special projective varieties},
   series={Lecture Notes of the Unione Matematica Italiana},
   volume={18},
   publisher={Springer, Cham; Unione Matematica Italiana, Bologna},
   date={2016},
   pages={xxvi+232},
   isbn={978-3-319-26764-7},
   isbn={978-3-319-26765-4},
   doi={10.1007/978-3-319-26765-4},
}



\bib{TS}{article}{
   author={Sebastiani, M.},
   author={Thom, R.},
   title={Un r\'{e}sultat sur la monodromie},
   language={French},
   journal={Invent. Math.},
   volume={13},
   date={1971},
   pages={90--96},
   issn={0020-9910},
   doi={10.1007/BF01390095},
}

\bib{Seg}{book}{
   author={Segre, B.},
   title={The Non-singular Cubic Surfaces},
   publisher={Oxford University Press, Oxford},
   date={1942},
   pages={xi+180},
}

\bib{Tho}{article}{
   author={Thom, R.},
   title={Probl\`emes rencontr\'{e}s dans mon parcours math\'{e}matique: un
   bilan},
   language={French},
   journal={Inst. Hautes \'{E}tudes Sci. Publ. Math.},
   number={70},
   date={1989},
   pages={199--214 (1990)},
   issn={0073-8301},
}




\end{thebibliography}

\end{document}